\documentclass[11pt]{article}

\usepackage[utf8]{inputenc}
\usepackage[T1]{fontenc}
\usepackage{multicol}
\usepackage{amsmath}
\usepackage{amsthm}
\usepackage{amssymb}
\usepackage{mathtools}
\usepackage{dsfont}

\usepackage{thmtools, thm-restate}

\usepackage{bbold}
\usepackage{xparse}
\usepackage{physics}
\usepackage{empheq}
\usepackage{url}
\usepackage{authblk} %
\usepackage{enumitem}
\usepackage{tikz}
\usetikzlibrary{quotes,angles,calc}
\usepackage{rotating}
\usepackage{graphicx}
\usepackage[linesnumbered,ruled,vlined]{algorithm2e}
\usepackage{mathrsfs}

\usepackage{pgfplots}
\pgfplotsset{compat = newest}

\usepackage[top=3cm,bottom=2cm,right=2cm,left=2cm]{geometry}

\newtheorem{theorem}{Theorem}[section]
\newtheorem{lemma}[theorem]{Lemma}

\newtheorem{proposition}[theorem]{Proposition}
\newtheorem{claim}[theorem]{Claim}
\newtheorem{observation}[theorem]{Observation}
\theoremstyle{definition}
\newtheorem{definition}[theorem]{Definition}

\SetArgSty{textup}
\DontPrintSemicolon
\SetKw{Continue}{continue}
\SetKwInOut{Input}{input}
\SetKwInOut{Output}{output}

\usepackage[colorinlistoftodos]{todonotes} %
\usepackage[colorlinks=true, allcolors=blue]{hyperref}
\newif\ifnotes\notestrue %
\ifnotes

\newcommand{\notename}[2]{{\textcolor{red}{\footnotesize{\bf (#1:} {#2}{\bf ) }}}}

\newcommand{\cnote}[1]{{\notename{Cedric}{#1}}}
\newcommand{\enote}[1]{{\notename{Edin}{#1}}}
\newcommand{\lnote}[1]{{\notename{Laci}{#1}}}
\newcommand{\gnote}[1]{{\notename{Georg}{#1}}}
\renewcommand{\b}[1]{{\color{blue} #1}}
\else
\newcommand{\notename}[2]{{}}

\newcommand{\cnote}[1]{}
\newcommand{\enote}[1]{}
\newcommand{\lnote}[1]{}
\newcommand{\gnote}[1]{}
\fi

\newcommand{\hf}{\hat{f}}
\newcommand{\tf}{\tilde{f}}

\newcommand{\E}{{{\mathbb{E}}}}
\newcommand{\R}{{{\mathbb{R}}}}
\newcommand{\Q}{{{\mathbb{Q}}}}
\newcommand{\N}{{{\mathbb{N}}}}
\newcommand{\Z}{{{\mathbb{Z}}}}
\newcommand{\1}{{{\mathbb{1}}}}
\newcommand{\0}{{{\mathbb{0}}}}
\newcommand{\lowerp}{\mu}
\newcommand{\cJ}{{{\mathcal{J}}}}
\newcommand{\cI}{{{\mathcal{I}}}}
\newcommand{\cP}{{{\mathcal{P}}}}
\newcommand{\cM}{{{\mathcal{M}}}}
\newcommand{\CG}{{{\mathcal{CG}}}}
\newcommand{\supp}{{{\mathrm{supp}}}}

\newcommand{\set}[1]{\ensuremath{\left\{#1\right\}}}
\newcommand{\pr}[1]{\ensuremath{\left(#1\right)}}
\newcommand{\br}[1]{\ensuremath{\left[#1\right]}}

\newcommand{\floor}[1]{\ensuremath{\left\lfloor#1\right\rfloor}}
\newcommand\SetOf[2]{\left\{#1 \mid #2\right\}}

\DeclareMathOperator*{\argmin}{arg\,min}
\DeclareMathOperator{\Poi}{Poi}

\DeclareMathOperator{\Bin}{Bin}

\newcommand{\pind}{\cP} %

\newcommand{\mgir}{\gamma} %
\newcommand{\rgir}{\ell} %

\newcommand{\rk}{r} %
\newcommand{\Rk}{R} %
\newcommand{\mrk}{\rho} %
\newcommand{\xrk}{\lambda} %
\newcommand{\rpiece}{\psi} %
\newcommand{\conc}{\rho} %
\newcommand{\finitediff}{\phi} %

\newcommand{\PP}{\theta} %
\newcommand{\QS}{Q}

\colorlet{darkgreen}{green!40!black}

\title{On the Correlation Gap of Matroids\thanks{This project has received funding from the European Research Council (ERC) under the European Union’s Horizon 2020 research and innovation programme (grant agreement nos.~757481--ScaleOpt and 805241--QIP). Part of this work was done while ZKK, GL and LAV participated in the Discrete Optimization Trimester Program at the Hausdorff Institute for Mathematics in Bonn in 2021. An extended abstract of this paper has appeared in Proceedings of the 24rd Conference on Integer Programming and Combinatorial Optimization, IPCO 2023.}}
\date{{\tt edin.husic@supsi.ch, zhuan.koh@cwi.nl, g.loho@utwente.nl, l.vegh@lse.ac.uk}}

\author[1]{Edin Husi\'{c}}
\author[2]{Zhuan Khye Koh}
\author[3]{Georg Loho} 
\author[4]{L{\'{a}}szl{\'{o}} A. V{\'{e}}gh}
\affil[1]{IDSIA, USI-SUPSI, Switzerland}
\affil[2]{Centrum Wiskunde \& Informatica, The Netherlands.} 
\affil[3]{University of Twente, The Netherlands}
\affil[4]{London School of Economics and Political Science, UK} 

\begin{document}

\maketitle

\begin{abstract}
A set function can be extended to the unit cube in various ways; the correlation gap measures the ratio between two natural extensions. 
This quantity has been identified as the performance guarantee in a range of  approximation algorithms and mechanism design settings. 
It is known that the correlation gap of a monotone submodular function is at least $1-1/e$, and this is tight for simple matroid rank functions.

We initiate a fine-grained study of the correlation gap of matroid rank functions. In particular, we present an improved lower bound on the correlation gap as parametrized by the rank and girth of the matroid. We also show that for any matroid, the correlation gap of its weighted rank function is minimized under uniform weights.
Such improved lower bounds have direct applications for submodular maximization under matroid constraints, mechanism design, and contention resolution schemes. 
\end{abstract}

\section{Introduction}
\label{sec:intro}
A continuous function $h \colon [0,1]^E\to \R_+$ is an \emph{extension} of a set function $f \colon 2^E\rightarrow \R_+$ if for every $x\in [0,1]^E$, $h(x)=\E_\lambda[f(S)]$ where $\lambda$ is a probability distribution over $2^E$ with marginals $x$, i.e.~$\sum_{S: i\in S}\lambda_S = x_i$ for all $i\in E$. 
Note that this in particular implies
$f(S)=h(\chi_S)$ for every $S\subseteq E$, where $\chi_S$ denotes the $0$-$1$ indicator vector of $S$. 

Two natural extensions are the following. The first one 
 corresponds to sampling each $i\in E$ independently with probability $x_i$, i.e., $\lambda_S=\prod_{i\in S}x_i\prod_{i\notin S}(1-x_i)$. Thus, 
\begin{equation}
\label{eq:multilinear-extension}
F(x)=\sum_{S\subseteq E}f(S)\prod_{i\in S}x_i\prod_{i\notin S}(1-x_i) \enspace .
\end{equation}
This is known as the \emph{multilinear extension} in the context of submodular optimization, see \cite{conf/ipco/CalinescuCPV07}.
The second extension corresponds to the probability distribution with maximum expectation:
\begin{equation}
\label{eq:concave-extension}
\hf(x)=\max\left\{\sum_{S\subseteq E} \lambda_S f(S):\, \sum_{S\subseteq E: i\in S} {\lambda_S}=x_i\, \,\forall i\in E\, ,\, 
 \sum_{S\subseteq E} {\lambda_S}=1\, ,\,  \lambda\ge 0\right\} \enspace .
\end{equation}
Equivalently, $\hf(x)$ is the upper part of the convex hull of the graph of $f$; we call it the \emph{concave extension} following the terminology of discrete convex analysis \cite{journals/oms/Murota21a}. 

Agrawal, Ding, Saberi and Ye   \cite{Agrawal2012} introduced the \emph{correlation gap} as the worst case ratio
\begin{equation}\label{eq:CG}
\CG(f):=\inf_{x\in [0,1]^E} \frac{F(x)}{\hf(x)} \enspace ,
\end{equation}
with the convention $0/0=1$.
It captures the maximum gain achievable by allowing correlations as opposed to independently sampling the variables.
This ratio plays a fundamental role in stochastic optimization~\cite{Agrawal2012,nikolova2010approximation},
mechanism design~\cite{bhalgat2012mechanism,hartline2013mechanism,yan2011mechanism}, 
prophet inequalities~\cite{conf/sagt/ChekuriL21,conf/innovations/Dughmi22,rubinstein2017combinatorial},
and a variety of submodular optimization problems~\cite{conf/sigecom/AsadpourNSS22,chekuri2014submodular}.

\medskip

The focus of this paper is on weighted matroid rank functions. %
For a matroid $\mathcal{M} = (E,\mathcal{I})$ and a weight vector $w\in \R^E_+$, the corresponding \emph{weighted matroid rank function} is given by
\begin{equation} \label{eq:weighted-matroid-rank}
  r_w(S) \coloneqq \max\{w(T):T\subseteq S,T\in \mathcal{I}\}.
\end{equation}
It is monotone nondecreasing and submodular. Recall that a function $f \colon 2^E \to \R_+$ is \emph{monotone} if $f(X)\leq f(Y)$ for all $X\subseteq Y\subseteq E$, and \emph{submodular} if $f(X)+f(Y)\ge f(X\cap Y)+f(X\cup Y)$ for all $X,Y\subseteq E$.

\subsection{The Role of Correlation Gap}
The correlation gap of a weighted matroid rank function has been identified as the performance guarantee in a range of approximation algorithms and mechanism design settings.
In what follows, we give an overview of three main applications where the correlation gap shows up as a critical parameter. 

\subsubsection{Monotone Submodular Maximization}
Let $f \colon 2^E\rightarrow \R_+$ be a monotone submodular function, and let $(E,\cJ)$ be a matroid with independent sets~$\cJ$.
We consider the problem of maximizing $f$ subject to a matroid constraint,
\begin{equation}\label{prob:submod-max}
\max_{S\in \cJ} f(S) \enspace .
\end{equation}
For uniform matroids (i.e., cardinality constraints), a classical result by Fisher, Nemhauser, and Wolsey~\cite{journals/mp/NemhauserWF78} showed a $(1-1/e)$-approximation  guarantee for  the greedy algorithm. Moreover,  
this factor cannot be improved if we are only allowed polynomially many calls to the value oracle of $f$ (see Nemhauser and Wolsey \cite{journals/mor/NemhauserW78}). 

The factor $(1-1/e)$ being the natural target for \eqref{prob:submod-max}, 
Calinescu, Chekuri, P\'al, and Vondr\'ak \cite{conf/ipco/CalinescuCPV07} obtained it for the special case when $f$ is a sum of weighted matroid rank functions.
The journal version of the same paper \cite{journals/siamcomp/CalinescuCPV11} (see also \cite{Vondrak2008}), shows a $(1-1/e)$-approximation for arbitrary monotone submodular functions, achieving the best possible general guarantee for \eqref{prob:submod-max}. 

These algorithms proceed in two steps. Let 
\begin{equation}\label{eq:matroid-polytope}
  \pind(\rk) \coloneqq \SetOf{x\in \R^E_+}{x(S)\le \rk(S)\quad\forall S\subseteq E}
\end{equation} denote the independent set polytope, where $\rk$ is the rank function of the matroid $(E,\cJ)$. When clear from the context, we use the shorthand $\pind$. 

In the first step, the goal is to find
a $(1-1/e)$-approximation algorithm for the relaxation
\begin{equation}\label{prob:multilinear-max}
\max_{x\in \pind} F(x) \enspace .
\end{equation}
Let $x^*$ be the  solution obtained in the first step. In the second step, they use \emph{pipage rounding} to find an integer solution $X\in\cJ$ with $f(X)\ge F(x^*)$.

Thus, approximation loss only happens in the first step. To solve this non-concave  maximization problem, \cite{conf/ipco/CalinescuCPV07} introduced another relaxation $\tf(x)$ such that $F(x)\le \tf(x)\le \hf(x)$ for all $x\in [0,1]^E$, and showed that $\max_{x\in \pind} \tf(x)$ can be formulated as an LP. The  $(1-1/e)$-approximation to \eqref{prob:multilinear-max} is obtained by solving this LP optimally.
Subsequently, Shioura~\cite{journals/dmaa/Shioura09} showed that when $f$ is a sum of monotone $M^\natural$-concave functions, the analogous convex program can also be solved optimally.
$M^\natural$-concave functions  form a special class of submodular functions, and are a central concept in discrete convex analysis (see Murota's monograph~\cite{Murotabook}). They are also known as \emph{gross substitutes functions} and play an important role in mathematical economics~\cite{gul1999walrasian,kelso1982job,leme2017gross,DBLP:books/cu/NRTV2007}.
We remark that every weighted matroid rank function is $M^\natural$-concave.

The 
$(1-1/e)$-approximation for arbitrary monotone submodular $f$ in \cite{journals/siamcomp/CalinescuCPV11,Vondrak2008} uses a different approach: instead of using another relaxation, they perform
 a \emph{continuous greedy algorithm} directly on $F(x)$. 
Improved approximations were subsequently given for submodular functions with bounded curvature; we discuss these results in Section~\ref{sec:further}.

\paragraph{Concave Coverage Problems} 
Let us focus on Shioura's \cite{journals/dmaa/Shioura09} specialization of problem \eqref{prob:submod-max}, i.e., $f=\sum_{i=1}^m f_i$ where each $f_i$ is monotone $M^\natural$-concave.
A basic example of this model is the \emph{maximum coverage} problem.
Given $m$ subsets $E_i\subseteq E$, the corresponding \emph{coverage function} is defined as $f(S)=|\{i\in[m]: E_i\cap S\neq\emptyset\}|$.
Note that this is a special case of a sum of weighted matroid rank functions:
$f(S)=\sum_{i=1}^m f_i(S)$ where $f_i(S)$ is the rank function of a rank-1 uniform matroid with support $E_i$. Even for maximization under a cardinality constraint, there is no better than $(1-1/e)$-approximation for this problem unless $P=NP$ (see Feige \cite{Feige1998}).

Recently, tight approximations have been established for another special case when the function values $f_i(S)$ are determined by the cardinality of the set $S$ and the support of $E_i$. 
Barman, Fawzi, and Ferm\'e \cite{conf/stacs/BarmanFF21} studied the \emph{maximum concave coverage} problem: given a monotone concave function $\varphi \colon \Z_+\to\R_+$ and weights $w\in \R_+^m$, the submodular function is defined as
$f(S)=\sum_{i=1}^m w_i \varphi(|S\cap E_i|)$.\footnote{We note that such functions are exactly the one-dimensional monotone $M^\natural$-concave functions $f_i \colon \Z_+ \to \R_+$.} The maximum coverage problem corresponds to $\varphi(x)=\min\{1,x\}$; on the other extreme, for  $\varphi(x)=x$ we get the trivial problem $f(S)=\sum_{j\in S}|\{i\in [m]: j\in E_i\}|$. In \cite{conf/stacs/BarmanFF21}, they present a tight approximation guarantee for maximizing such an objective subject to a matroid constraint,  parametrized by the \emph{Poisson curvature} of the function~$\varphi$.

This extends previous work by Barman, Fawzi, Ghoshal, and G\"urpinar~\cite{journals/mp/BarmanFGG22} which considered $\varphi(x)=\min\{\ell,x\}$ 
(for $\ell > 1$), motivated by the list decoding problem in coding theory. 
It also generalizes the work by Dudycz, Manurangsi, Marcinkowski, and Sornat~\cite{Dudycz2020} 
which considered geometrically dominated concave functions $\varphi$, motivated by approval voting rules such as Thiele rules, proportional approval voting, and $p$-geometric rules. 
In both cases, the obtained approximation guarantees improve over the $1-1/e$ factor. 

\medskip

In Section~\ref{sec:max-corr}, we make the simple observation that the algorithm of Calinescu et al.~\cite{conf/ipco/CalinescuCPV07} and Shioura~\cite{journals/dmaa/Shioura09} actually has an approximation ratio of $\min_{i\in [m]}\CG(f_i)$.
For technical reasons, we assume that all the $f_i$'s are rational-valued; the relevant complexity parameter $\lowerp(f)$ is defined in Section~\ref{sec:prelim}.

\begin{proposition}\label{prop:sum-mconcave}
Let $f_1,f_2,\ldots,f_m \colon 2^E\to\R$ be monotone $M^\natural$-concave functions, and let $f=\sum_{i=1}^m f_i$. Then, a 
 $\min_{i=1}^m \CG(f_i)$-approximation algorithm for \eqref{prob:submod-max} can be obtained in time polynomial in $|E|$, $m$ and~$\lowerp(f)$. %
\end{proposition}

We also prove that the Poisson curvature of $\varphi$ is essentially the correlation gap of the functions $\varphi(|S\cap E_i|)$.
Hence, the approximation guarantees in \cite{conf/stacs/BarmanFF21,journals/mp/BarmanFGG22,Dudycz2020} are in fact correlation gap bounds, and they can be derived from Proposition~\ref{prop:sum-mconcave} via a single unified algorithm, i.e., the one by Calinescu et al.~\cite{conf/ipco/CalinescuCPV07} and Shioura~\cite{journals/dmaa/Shioura09}.
In particular, the result of Barman et al.~\cite{journals/mp/BarmanFGG22} which concerned $\varphi(x)=\min\{\ell,x\}$ (for $\ell>1$) boils down to the analysis of uniform matroid correlation gaps.

\subsubsection{Sequential Posted Price Mechanisms}
Following Yan~\cite{yan2011mechanism}, consider a seller with a set of identical services (or goods), and a set $E$ of agents where each agent is only interested in one service (unit demand). 
Agent $i\in E$ has a private valuation $v_i$ if they get a service and 0 otherwise,
where each $v_i$ is drawn independently from a known distribution $F_i$ over $[0, L]$ for large $L \in \R_+$ with positive smooth density function. 
The seller can offer the service only to certain subsets of the agents simultaneously; 
this is captured by a matroid $(E, \cI)$ where the independent sets represent feasible allocations of the services to the agents.

A mechanism uses an
allocation rule $x \colon \R_+^E \to \{0, 1\}^E$ 
to choose the winning set of agents based on the reported valuations $v \in \R_+^E$ of the agents, 
and uses a payment rule $p \colon \R_+^E \to \R_+^E$ to charge
the agents. 

Myerson’s mechanism~\cite{bulow1989simple,myerson1981optimal} guarantees the optimal revenue, but is highly intricate and there has been significant interest in simpler mechanisms such as sequential posted price mechanisms proposed by Chawla, Hartline, Malec, and Sivan~\cite{DBLP:conf/stoc/ChawlaHMS10}.

For a given ordering of agents and a price $p_i$ for each
agent $i$, a \emph{sequential posted-price mechanism (SPM)}
initializes the allocated set $A$ to be $\emptyset$, 
and for all agents $i$ in the order, does the following: if serving
$i$ is feasible, i.e., $A \cup \{i\} \in \cI$, then it offers to serve agent $i$ at
the pre-determined price $p_i$, and adds $i$ to $A$ if agent $i$ accepts. 

Thus, the seller makes take-it-or-leave-it price offers to agents one by one. 
This type of mechanism is easy to run for the sellers, limits agents' strategic behaviour, and
keeps the information elicited from agents at a
minimum level.
Simplicity comes at a cost as it does not deliver optimal revenue, 
but as it turns out, this cost can be lower bounded by the correlation gap of the underlying matroid $(E, \cI)$.

\begin{theorem}[{\cite[Theorem 3.1]{yan2011mechanism}}]\label{thm:SPM}
If the correlation gap of
the weighted rank function is at least $\beta$ for no matter
what nonnegative weights, then the expected revenue of 
greedy-SPM is a $\beta$-approximation to that of Myerson’s
optimal mechanism.    
\end{theorem}

Similarly, the same paper~\cite{yan2011mechanism} shows that a greedy-SPM mechanism recovers a constant factor of
the VCG mechanism~\cite{clarke1971multipart,groves1973incentives,vickrey1961counterspeculation} that maximizes the optimal welfare instead of revenue. The factor here is also the correlation gap of the (weighted) rank function of the underlying matroid.
The analysis of greedy-SPM in both revenue and welfare maximization settings is tight. For details we refer to~\cite{yan2011mechanism}.

\subsubsection{Contention Resolution Schemes}
Chekuri, Vondr\' ak, and Zenklusen~\cite{chekuri2014submodular} introduced contention resolution (CR) schemes
as a tool for maximizing a general submodular function $f$ (not necessarily monotone) under various types of constraints. %
For simplicity, let us illustrate it for a single matroid constraint, i.e.~\eqref{prob:submod-max} without the monotonicity assumption on $f$.
It consists of first approximately solving the continuous problem~\eqref{prob:multilinear-max}.
After obtaining an approximately optimal solution $x\in \cP$ to \eqref{prob:multilinear-max}, it is rounded to an integral and feasible solution --- i.e.~an independent set in $\cJ$ ---
 without losing too much in the objective value. 
At a high level, given a fractional point $x\in \cP$, a CR scheme first generates a random set $R(x)$
by independently including each element $i$ with probability $x_i$. 
Then, it removes some elements of  $R(x)$ to obtain a feasible solution. 
We say that a CR scheme is $c$-\emph{balanced} if, conditioned on $i\in R(x)$, the element $i$ is contained in the final independent set with probability at least $c$; see~\cite{chekuri2014submodular} for a  formal definition.
A $c$-balanced scheme delivers an integer solution with expected cost at least $cF(x)$.
Thus, the goal is to design $c$-balanced CR schemes with the highest possible value of $c$.

There is a tight relationship between CR schemes and the correlation gap.
Namely, the correlation gap of the weighted rank function of $(E,\mathcal{J})$ is equal to 
the maximum $c$ such that a $c$-balanced CR scheme exists~\cite[Theorem 4.6]{chekuri2014submodular}.
We would like to point out that the correlation gap here concerns the matroid in the constraint, unlike in Proposition \ref{prop:sum-mconcave} where the correlation gap concerns the objective function.

The benefit of CR schemes is that we can obtain good guarantees for submodular function maximization under an intersection of different (downward-closed) constraints, including multiple matroid constrains, knapsack constraints, matching etc. 
Moreover, in this case the CR scheme can be simply obtained by combining the CR schemes for individual constraints.\footnote{
We note however that CR schemes are not optimal for rounding~\eqref{prob:multilinear-max}: for this particular case, pipage or swap rounding finds a feasible integer solution of value $F(x)$, without any loss.}

\subsection{Our Results}

Motivated by the significance of correlation gap in algorithmic applications, we study the correlation gap of weighted matroid rank functions. 
It is well-known that $\CG(f) \geq 1-1/e$ for every monotone submodular function $f$ \cite{conf/ipco/CalinescuCPV07}. 
Moreover, the extreme case $1-1/e$ is already achieved by the rank function of a rank-1 uniform matroid as $|E| \to \infty$.
More generally, the rank function of a rank-$\ell$ uniform matroid has correlation gap $1-e^{-\ell}\ell^\ell/\ell! \geq 1-1/e$ \cite{yan2011mechanism,journals/mp/BarmanFGG22}. 
Other than for uniform matroids, we are not aware of any previous work that gave better than $1-1/e$ bounds on the correlation gap of specific matroids.

First, we show that among all weighted rank functions of a matroid, the smallest correlation gap is realized by its (unweighted) rank function.

\begin{theorem}\label{thm:weighted}
  For any matroid $\mathcal{M} = (E,\mathcal{I})$ with rank function $r=r_{\1}$,
  \[\inf_{w\in \R^E_+} \CG(r_w) = \CG(r).\]
\end{theorem}

For the purpose of lower bounding $\CG(r_w)$, Theorem \ref{thm:weighted} allows us to ignore the weights $w$ and just focus on the matroid $\cM$. As an application, to bound the approximation ratio of sequential posted-price mechanisms as in Theorem \ref{thm:SPM}, it suffices to focus on the underlying matroid.

We remark that $\mathcal{M}$ can be assumed to be \emph{connected}, that is, it cannot be written as a direct sum of at least two nonempty matroids.
Otherwise, $r=\sum_{i=1}^k r_i$ for matroid rank functions $r_i$ with disjoint supports.
It follows that the concave and multilinear extensions of $r$ can be written as $\hat{r} = \sum_{i=1}^k \hat{r}_i$ and $R = \sum_{i=1}^k R_i$ respectively.
Hence, $\CG(r)\geq \min_{i\in [k]}\CG(r_i)$ by the mediant inequality.
As the reverse inequality holds trivially, we have $\CG(r) = \min_{i\in [k]}\CG(r_i)$.
For example, the correlation gap of a partition matroid is equal to the smallest correlation gap of its parts (uniform matroids).

\begin{proposition}\label{prop:direct-sum}
Let $\mathcal{M}$ be a matroid with rank function $r$.
If $\mathcal{M} = \mathcal{M}_1 \oplus \cdots \oplus \mathcal{M}_k$ where each $\mathcal{M}_i$ is a matroid with rank function $r_i$, then $\CG(r) = \min_{i\in [k]} \CG(r_i)$.
\end{proposition}

Our goal is to identify the parameters of a matroid which govern its correlation gap. 
A natural candidate is the rank $r(E)$. However, as pointed out by Yan \cite{yan2011mechanism}, there exist matroids with arbitrarily high rank whose correlation gap is still $1-1/e$, e.g., partition matroids with rank-1 parts. 
The $1-e^{-\ell}\ell^\ell/\ell!$ bound for uniform matroids \cite{yan2011mechanism,journals/mp/BarmanFGG22} is suggestive of girth as another potential candidate. 
Recall that the \emph{girth} of a matroid is the smallest size of a dependent set.
On its own, a large girth does not guarantee improved correlation gap bounds: we show that for any $\gamma\in \N$, there exist matroids with girth $\gamma$ whose correlation gaps are arbitrarily close to $1-1/e$ (Proposition \ref{prop:upper}). 

It turns out that the correlation gap heavily depends on the relative values of the rank and girth of the matroid.
Our second result is an improved lower bound on the correlation gap as a function of these two parameters. 
\begin{theorem}\label{thm:monster}
Let $\cM=(E,\cI)$ be a loopless matroid with rank function $r$, rank $r(E)=\mrk$, and girth $\mgir$.
Then, 
\begin{align*}
 \CG(r)\ge  1-\frac{1}{e} + \frac{e^{-\mrk}}{\mrk} \pr{\sum_{i=0}^{\mgir-2} (\mgir-1-i) \br{{\mrk \choose i}(e-1)^i - \frac{\mrk^i}{i!}}}\ge 1-\frac{1}{e}\, .
\end{align*}
Furthermore, the last inequality is strict whenever $\gamma>2$.
\end{theorem}

Figure \ref{fig:corgap} illustrates the behaviour of the expression in Theorem \ref{thm:monster}.
For any fixed girth $\gamma$, it is monotone decreasing in $\rho$ (Lemma~\ref{lem:final-monotone}).
On the other hand, for any fixed rank~$\rho$, it is monotone increasing in $\gamma$ 
(Lemma~\ref{lem:monotoneFixedRank}).
In Section \ref{sec:upper}, we also give complementing albeit non-tight upper bounds that behave similarly with respect to these parameters.
When $\rho = \gamma-1$, our lower bound simplifies to $1-e^{-\rho}\rho^\rho/\rho!$, i.e., the correlation gap of a rank-$\rho$ uniform matroid (Proposition~\ref{prop:uniform-bound}).

\begin{figure}[h]\label{fig:corgap}
    \raggedright
    \includegraphics[width=0.45\textwidth]{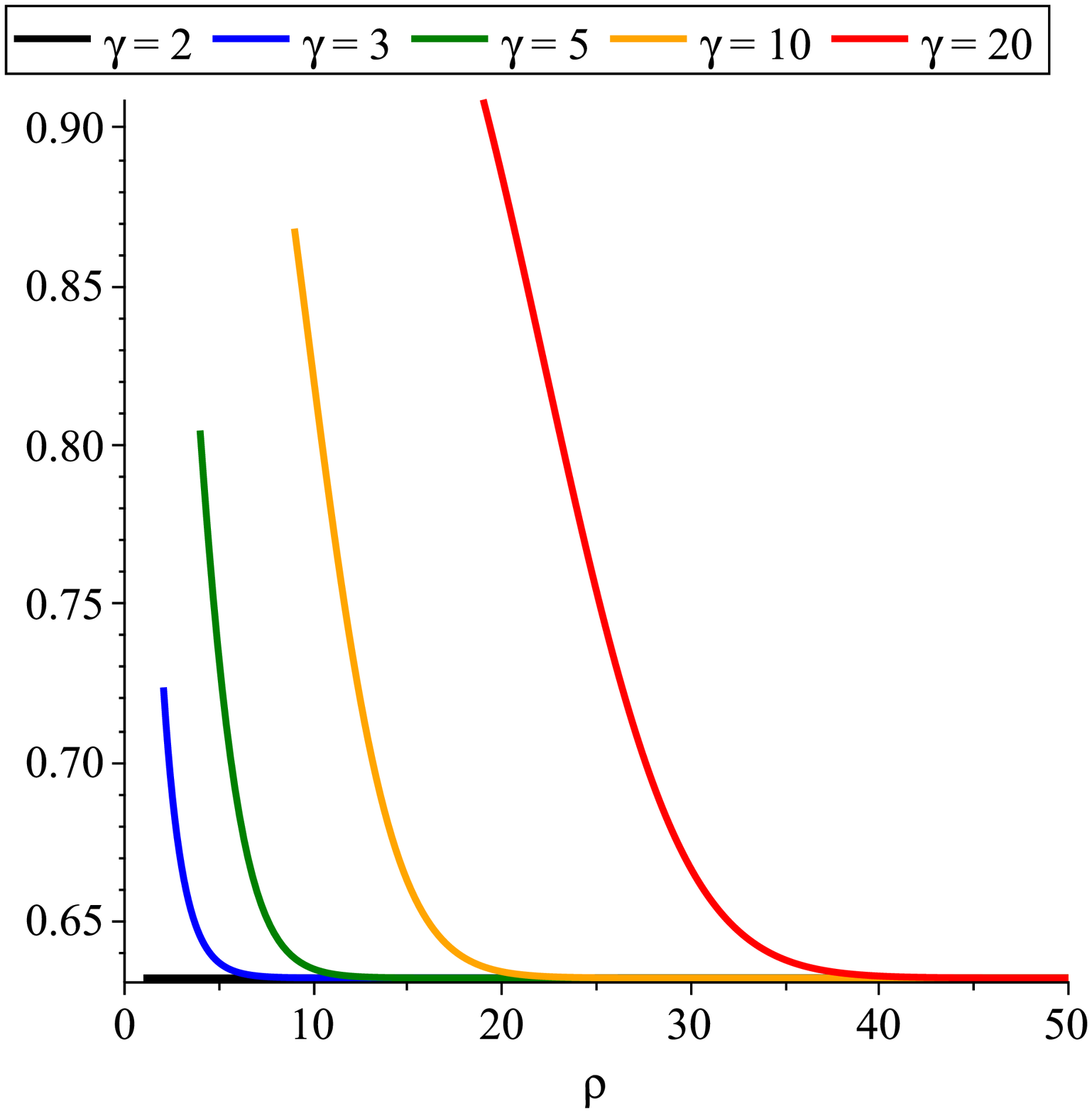}
    \hfill
    \includegraphics[width=0.45\textwidth]{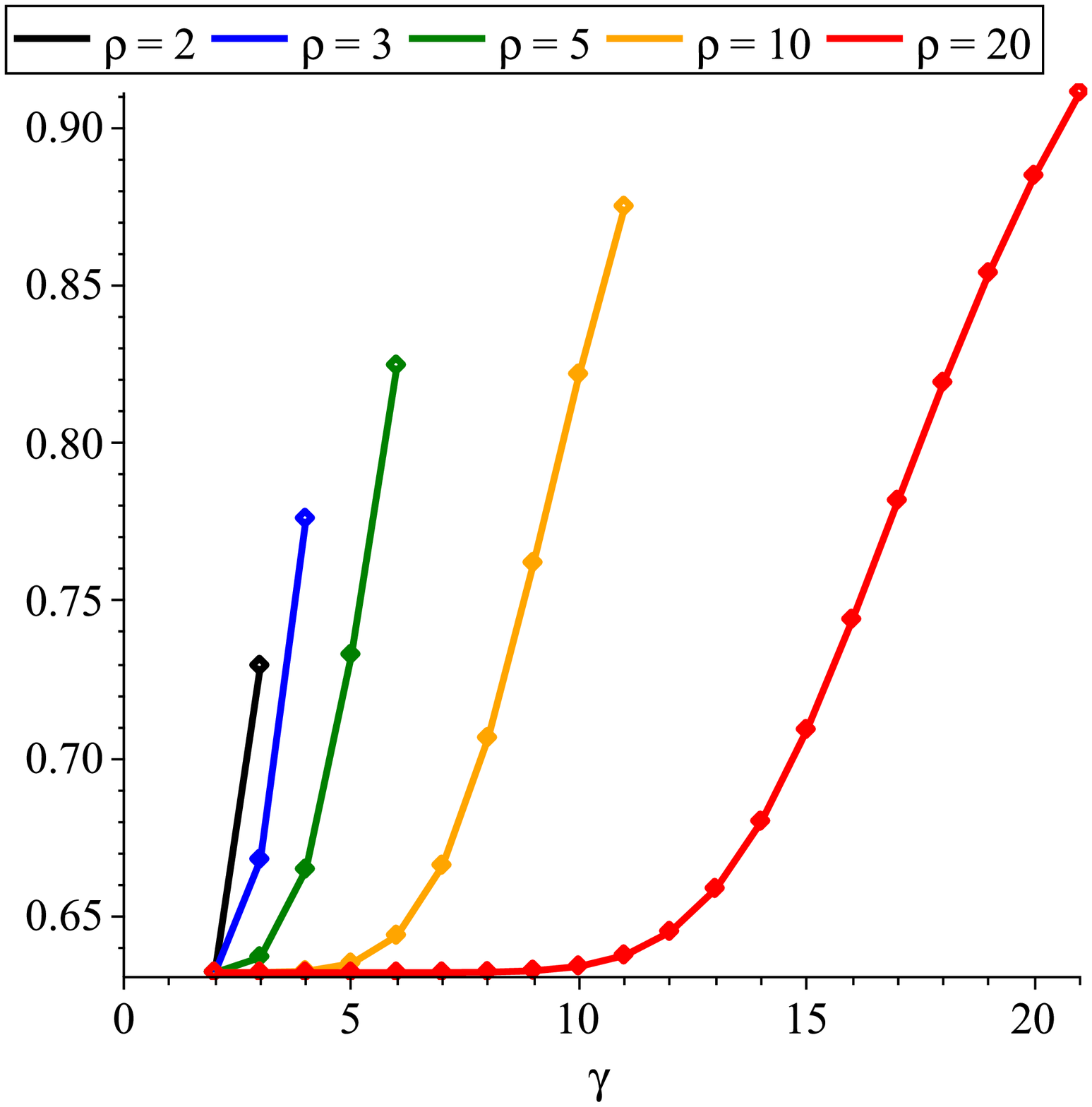}
    \caption{Our correlation gap bound as a function of the rank $\rho$, and as a function of the girth $\gamma$ separately.}
\end{figure}

The rank and girth have meaningful interpretations in the aforementioned applications.
For instance, consider the problem of maximizing a sum of weighted matroid rank functions $\sum_{i=1}^m f_i$ under a matroid constraint $(E,\cJ)$.
For every $i\in[m]$, let $\cM_i$ be the matroid of $f_i$.
In game-theoretic contexts, each $f_i$ usually represents the utility function of agent $i$.
Thus, our goal is to select a bundle of items $S\in \cJ$ which maximizes the total welfare.
If $\cM_i$ has girth $\gamma$ and rank $\rho$, this means that agent $i$ is interested in $\gamma-1\leq k \leq \rho$ items with positive weights.
The special case $\rho = \gamma-1$ (uniform matroids) has already found applications in list decoding \cite{journals/mp/BarmanFGG22} and approval voting \cite{Dudycz2020}.
On the other hand, for sequential posted-price mechanisms, if the underlying matroid $\mathcal{M}$ has girth $\gamma$ and rank $\rho$, this means that the seller can service $\gamma-1\leq k\leq \rho$ agents simultaneously.

To the best of our knowledge, our results give the first improvement over the $(1-1/e)$ bound on the correlation gap of general matroids.
We hope that our paper will motivate further studies into more refined correlation gap bounds, exploring the dependence on further matroid parameters, as well as obtaining tight bounds for special matroid classes.

\subsection{Our Techniques}

We now give a high-level overview of the proofs of Theorem~\ref{thm:weighted} and Theorem~\ref{thm:monster}.

\paragraph{Weighted Rank Functions}
The first step in proving both Theorem~\ref{thm:weighted} and Theorem~\ref{thm:monster} is to deduce structural properties of the points which realize the correlation gap.
In Theorem~\ref{thm:locating}, we show that such a point $x$ can be found in the independent set polytope $\cP$.
This implies that $\hat r_w(x)=w^\top x$ for any weights $w\in\R^E_+$. 
Moreover, we deduce that $x(E)$ must be integral.

To prove Theorem~\ref{thm:weighted}, we fix a matroid $\cM$ and derive a contradiction for a non-uniform weighting. 
More precisely, we consider a weighting $w\in \R^E_+$ and a point $x^*\in [0,1]^E$ which give a smaller ratio $R_w(x^*)/\hat r_w(x^*)<\CG(r)$. 
By the above, we can use the simpler form $R_w(x^*)/\hat r_w(x^*)=R_w(x^*)/w^\top x^*$.
We pick $w$ such that it has the smallest number of different values.
If the number of distinct values is at least 2, then we derive a contradiction by showing that a better solution can be obtained by increasing the weights in a carefully chosen value class until they coincide with the next smallest value.
The greedy maximization property of matroids is essential for this argument.

\paragraph{Uniform Matroids}
Before outlining our proof of Theorem~\ref{thm:monster}, let us revisit the correlation gap of uniform matroids.
Let $\cM=(E,\cI)$ be a uniform matroid on $n$ elements with rank $\mrk=r(E)$.
If $\mrk=1$, then it is easy to verify that the symmetric point $x=(1/n)\cdot\1 $ realizes the correlation gap $1-1/e$.
Since $x$ lies in the independent set polytope, we have $\hat r(x)=\1^{\top}x=1$.
If one samples each $i\in E$ with probability $1/n$, the probability of selecting at least one element is $R(x) = 1-(1-1/n)^n$.
Thus, $\CG(r) = 1-(1-1/n)^n$, which converges to $1-1/e$ as $n\to\infty$.
More generally, for $\mrk\geq 1$, Yan \cite{yan2011mechanism} showed that the symmetric point $x=(\mrk/n)\cdot \1$ similarly realizes the correlation gap $1-e^{-\rho}\rho^\rho/\rho!$.

\paragraph{Poisson Clock Analysis}
To obtain the $(1-1/e)$ lower bound on the correlation gap of a monotone submodular function, 
Calinescu et al.~\cite{conf/ipco/CalinescuCPV07} introduced an elegant probabilistic analysis. Instead of sampling each $i\in E$ with probability $x_i $, they consider $n$ independent \emph{Poisson clocks} of rate $x_i$ that are active during the time interval $[0,1]$. Every clock may send at most one signal from a Poisson process.
Let $\QS(t)$ be the set of elements whose signal was sent between time $0$ and $t$; the output is $\QS(1)$.
It is easy to see that $\E[f(\QS(1))]\le F(x)$.

In \cite{conf/ipco/CalinescuCPV07}, they show that the derivative of $\E[f(\QS(t))]$ can be lower bounded as $f^*(x)-\E[f(\QS(t))]$ for every $t\in [0,1]$, where 
\begin{equation}\label{eq:marginal}
  f^*(x) \coloneqq \min_{S\subseteq E} \left(f(S) + \sum_{i\in E}f_S(i)x_i \right)
\end{equation}
is an extension of $f$ such that $f^*\geq \hat{f}$. 
Here, $f_S(i) \coloneqq f(S\cup\{i\}) - f(S)$ is the marginal gain of adding element $i$ to set $S$.
The bound 
 $\E[f(\QS(1))]\ge (1-1/e)f^*(x)$ is obtained by solving a differential inequality. Thus, $F(x)\ge \E[f(\QS(1))]\ge (1-1/e) f^*(x)\ge (1-1/e)\hat f(x)$ follows.

\paragraph{A Two Stage Approach}
If $f$ is a matroid rank function, then we have $f^*=\hat f$ (see Theorem~\ref{thm:box-intersect}). Still, the factor $(1-1/e)$ in the analysis cannot be improved: for an integer $x\in \mathcal{P}$, we lose a factor $(1-1/e)$ due to  $\E[f(\QS(1))]=(1-1/e)F(x)$, even though the extensions coincide:
$F(x)=\hat f(x)$.

Our analysis in Section~\ref{sec:gap-analysis} proceeds in two stages.
Let $\cM=(E,\cI)$ be a matroid with rank $\mrk$ and girth $\mgir$.
The basic idea is that up to sets of size $\gamma-1$, our matroid `looks like' a uniform matroid.
Since the correlation gap of uniform matroids is well-understood, we first extract a uniform matroid of rank $\gamma-1$ from our matroid, and then analyze the contribution from the remaining part separately.
More precisely, we decompose the rank function as $\rk=g+h$, where $g(S)=\min\{|S|,\ell\}$ is the rank function of a uniform matroid of rank $\rgir=\mgir-1$. 
Note that the residual function $h:=f-g$ is not submodular in general, as $h(S) = 0$ for all $|S|\leq \ell$. 
We will lower bound the multilinear extensions $G(x)$ and $H(x)$ separately. 
As $g$ is the rank function of a uniform matroid, we can lower bound $G(x)$ as above by showing that the minimum is achieved at a symmetric point, i.e., $x_i=x(E)/n$ for all $i\in E$.

Bounding $H(x)$ is based on a Poisson clock analysis as in \cite{conf/ipco/CalinescuCPV07}, but is significantly more involved.
Due to the monotonicity of $h$, directly applying the result in  \cite{conf/ipco/CalinescuCPV07} would yield $\E[h(\QS(1)] \geq (1-1/e)h^*(x)$.
However, $h^*(x) = 0$ whenever $\mathcal{M}$ is loopless ($\ell\geq 1$), as $h(\emptyset) = 0$ and $h(\{i\}) = 0$ for all $i\in E$.
So, the argument of \cite{conf/ipco/CalinescuCPV07} directly only leads to the trivial $\E[h(\QS(1))]\geq 0$.
Nevertheless, one can still show that the derivative of the conditional expectation $\E[h(\QS(t))||\QS(t)|\geq \ell]$ is at least $r^*(x) - \ell - \E[h(\QS(t))||\QS(t)|\geq \ell]$.
Let $T\geq 0$ be the earliest time such that $|\QS(T)|\geq \ell$, which we call the \emph{activation time} of $\QS$.
Then, solving a similar differential inequality produces $\E[h(\QS(1))|T= t] \geq (1-e^{-(1-t)})(r^*(x) - \ell)$ for all $t\leq 1$.

To lower bound $\E[h(\QS(1))]$, it is left to take the expectation over all possible activation times $T\in [0,1]$.
Let $\bar{h}(x) = (r^*(x)-\ell)\int_0^1 f_T(t)(1-e^{-(1-t)}) dt$ be the resulting expression, where $f_T$ denotes the probability density function of $T$.
We prove that $\bar{h}(x)$ is concave in each direction $e_i - e_j$ for $i,j\in E$.
This allows us to round $x$ to an integer $x'\in [0,1]^E$ such that $x'(E) = x(E)$ and $\bar{h}(x')\leq \bar{h}(x)$; recall that $x(E)\in \Z$ by Theorem \ref{thm:locating}.
After substantial simplification of $\bar{h}(x')$, we arrive at the formula in Theorem~\ref{thm:monster}, except that $\rho$ is replaced by $x(E)$.
So, the rounding procedure effectively shifts the dependency of the lower bound from the value of $x$ to the value of $x(E)$.
Since $x(E)\leq \mrk$ by Theorem \ref{thm:locating}, the final step is to prove that the formula in Theorem \ref{thm:monster} is monotone decreasing in $\mrk$.
This is shown in Lemma \ref{lem:final-monotone} using the relationship between the Poisson distribution and the incomplete gamma function.
Additionally, in Lemma~\ref{lem:nonnegative} we show that the obtained lower bound is always strictly greater than $1-1/e$ when $\ell>1$.

\subsection{Further Related Work}\label{sec:further}
In the context of submodular maximization \eqref{prob:submod-max}, Proposition~\ref{prop:sum-mconcave} allows for improved approximation bounds if $f=\sum_{i=1}^ m f_i$, where the $f_i$'s are  $M^\natural$-concave functions.

A different approach to give fine-grained approximation guarantees for \eqref{prob:submod-max} is via curvature notions; this is applicable to any submodular function and does not require the form 
$f=\sum_{i=1}^ m f_i$.
A well-studied measure is the \emph{total curvature} of the submodular function, namely, $c(f)=1-\min_{i\in E} (f(E)-f(E\setminus \{i\}))/f(\{i\})$. Monotonicity and submodularity guarantee $c(f)\in[0,1]$; the best case $c(f)=0$ corresponds to additive (modular) functions. For cardinality constraints, such a bound was given by Conforti and Cornu\'ejols \cite{journals/dam/ConfortiC84}, strengthened and extended to matroid constraints by Sviridenko, Vondr\'ak, and Ward \cite{journals/mor/SviridenkoVW17}.

However, there are important cases of submodular functions where the total curvature bound is not tight. For a nondecreasing concave univariate function $\varphi:\Z_+\to\R_+$ with $\varphi(0)=0$, $f(S)=\varphi(|S|)$ is a submodular function. Exact maximization over matroid constraints is straightforward for such a function, yet the total curvature can be 1. This is a simple example of an $M^\natural$-concave function; submodular function maximization can be done in polynomial time for all such functions (see Proposition~\ref{prop:submod-max-m-concave}).

Motivated by this, Soma and Yoshida \cite{conf/icalp/SomaY18} proposed the following generalization of total curvature: assume our monotone submodular function can be decomposed as $f=g+h$, where $g$ is monotone submodular and $h$ is $M^\natural$-concave. They define the $h$-curvature as $\gamma_h(f)=1-\min_{S\subseteq E}h(S)/f(S)$, and provide approximation guarantees in terms of $\gamma_h(f)$. If this is close to 0, then the function can be well-approximated by an $M^\natural$-concave function. The usual notion of total curvature arises by restricting $h$ to additive (modular) functions.

A common thread in \cite{conf/icalp/SomaY18}  and our approach is to exploit special properties of $M^\natural$-concave functions for submodular maximization. However, there does not appear to be any direct implication between them.

\paragraph{Paper Organization} %
In Section~\ref{sec:prelim}, we recall the definitions of matroids, $M^\natural$-concave functions, submodular functions, related paremeters, and some classical results that we will use in our proofs. 
In Section~\ref{sec:mnat-sum}, we recall Shioura's algorithm for maximizing a sum of $M^\natural$-concave functions under a matroid constraint, and observe that the performance of this algorithm is bounded by the correlation gap of the input functions.
Then, in Section~\ref{sec:barman}, we explain how the results on concave coverage problems~\cite{conf/stacs/BarmanFF21,journals/mp/BarmanFGG22,Dudycz2020} can be derived using the aforementioned algorithm, and that the Poisson curvature bound is essentially equal to the correlation gap of the same functions.

The rest of the paper is devoted to showing our two main results. 
In Section~\ref{sec:locating}, we prove Theorem~\ref{thm:weighted} and show that the minimizer of the correlation gap can always be found in the independent set polytope of the matroid. 
Then, we prove Theorem~\ref{thm:monster} in Section~\ref{sec:gap-analysis}.
Finally, we give upper-bounds on the correlation gap of a matroid with rank $\rho$ and girth $\gamma$ in Section~\ref{sec:upper}. 
Several lemmas and proofs are deferred to the appendix.

\section{Preliminaries}
\label{sec:prelim}
We denote $\Z_+$ and $\R_+$ as the set of nonnegative integers and nonnegative reals respectively. For  $n, k \in \Z_+$, $\binom{n}{k} = \frac{n!}{k!(n-k)!}$ if $n \geq k$, and $0$ otherwise. 
For a set $S$ and $i\in S$, $j\notin S$, we use the shorthand $S-i=S\setminus \{i\}$, $S+j=S\cup\{j\}$, and $S-i+j=(S\setminus\{i\})\cup \{j\}$.
For $x\in\R^E$ and $S\subseteq E$, we write $x(S)=\sum_{i\in S}x_i$.
For a set function $f$, the \emph{marginal gain} of adding an element $i$ to a set $S$ is denoted as $f_S(i) = f(S+i) - f(S)$.

All set functions in the paper will be given by value oracles; our running time bounds will be polynomial in the number of oracle calls and arithmetic operations. We further assume that all set functions are rational valued, and for $f:\, 2^E\to\Q$, we let $\lowerp(f)$ denote an upper bound on the encoding length of any value $f(S)$. That is, for any $S\subseteq E$, the oracle returns $f(S)=p/q$ represented by $p,q\in \Z$ such that $\lceil \log_2 |p|\rceil+\lceil \log_2 |q|\rceil\le \mu(f)$.

\paragraph{Matroids}
For a detailed introduction to matroids, we refer the reader to Oxley's book \cite{books/Oxley} or Schrijver's book \cite{Schrijverbook}.
A matroid $\mathcal{M}=(E,\mathcal{I})$  is given by a downward closed family of \emph{independent sets} $\cI\subseteq 2^E$ over a ground set $E$. We require that $\cI\neq\emptyset$, and the following axiom:
  \begin{equation} \label{eq:matroid}
      \forall X,Y \in \cI \text{ with } |X| < |Y|: \exists j\in Y\setminus X:\, X+j\in \cI\, . 
\end{equation}
A \emph{basis} is an inclusion-wise maximal independent set.
Let $\mathcal{B}\subseteq \mathcal{I}$ be the set of bases. The above axiom implies that all bases are of the same size, called the \emph{rank} of $\mathcal{M}$.

The rank function $\rk: 2^E\to\Z_+$ is defined as $\rk(S)=\max\{|T|:\, T\subseteq S, T\in\cI\}$. This is a monotone submodular function. 
A \emph{circuit} is an inclusion-wise minimal dependent set.
The size of a smallest circuit is called the \emph{girth} of $\mathcal{M}$.

Recall the independent set polytope defined in \eqref{eq:matroid-polytope}.
\begin{theorem}[{Edmonds, \cite[Theorem 40.2]{Schrijverbook}}]\label{thm:matroid-independent} For a matroid $\mathcal{M}=(E,\cI)$ with rank function $\rk$, $\pind(\rk)$ defined in \eqref{eq:matroid-polytope} is the convex hull of the incidence vectors of the independent sets in $\cI$.
\end{theorem}
We also recall another classical result by Edmonds on intersecting the independent set polytope by a box. 
\begin{theorem}[{Edmonds, \cite[Theorem 40.3]{Schrijverbook}}]\label{thm:box-intersect}
Let $r:\, 2^E\to\Z_+$ be a matroid rank function and $x\in\R_+$. Then,
\[
\max\{y(E):\, y\in\cP(r),\, y\le x\}=\min\{r(S)+x(E\setminus S): S\subseteq E\}\, .
\]
\end{theorem}

\paragraph{$M^\natural$-Concave Functions}
A set function $f:\, 2^E\to\R\cup\{-\infty\}$ is $M^\natural$-concave if
\begin{subequations}
  \begin{equation} \label{eq:Mnat-concave}
    \begin{aligned}
      &\forall X,Y \subseteq  \text{ with } |X| < |Y|:  \\
      &f(X) + f(Y) \leq \max_{j \in Y \setminus X} \{f(X + j) + f(Y - j)\}
    \end{aligned}
\end{equation}
  \begin{equation} \label{eq:M-concave}
    \begin{aligned}
      &\forall X,Y \subseteq E \text{ with } |X| = |Y| \text{ and } \forall i \in X \setminus Y: \\
      &f(X) + f(Y) \leq \max_{j \in Y \setminus X} \{f(X - i + j) + f(Y + i - j)\} .
    \end{aligned}
\end{equation}
\end{subequations}
We refer the reader to Murota's monography \cite{Murotabook} for a comprehensive treatment of $M^\natural$-concave functions. 
These functions can be defined more generally on the integer lattice $\Z^n$. In this paper, we restrict our attention to $M^\natural$-concave set functions, also known as \emph{valuated generalized matroids}. These are closely related to valuated matroids introduced by Dress and Wenzel \cite{DressWenzel:1992}.
 The definitions above are from~\cite{FujishigeHirai:2020,MurotaShioura:2018} and are equivalent to the standard definition in \cite{Murotabook}.

The definition can be seen as a generalization of the matroid independence axiom \eqref{eq:matroid}. Given a matroid $\cM=(E,\cI)$, the indicator function $f$ defined as $f(S)=0$ if $S\in\cI$ and $f(S)=-\infty$ is $M^\natural$-concave. More generally, given a weight vector $w\in\R^E_+$, the weighted matroid rank function as defined in \eqref{eq:weighted-matroid-rank} is $M^\natural$-concave. These functions form a nontrivial subclass of submodular functions~\cite{gul1999walrasian,leme2017gross}.
\begin{proposition}[{\cite[Theorem 6.19]{Murotabook}}]
Every $M^\natural$-concave function is submodular.
\end{proposition}
We recall that submodular functions can be minimized in polynomial time, but submodular maximization is NP-complete. However, it is polynomial time solvable for $M^\natural$-concave functions; in fact, they can be maximized using the greedy algorithm.
\begin{proposition}[\cite{DressWenzel:1992}]
If $f:\, 2^E\to\R\cup\{-\infty\}$ is an $M^\natural$-concave function, then for every vector $z\in \R^E$, $\max_{S\subseteq E} f(S)-z(S)$ can be computed in strongly polynomial time.
\end{proposition}
Recall  the concave extension $\hat f(x)$ defined in \eqref{eq:concave-extension}. It is NP-complete to evaluate this function for general submodular functions. However, for $M^\natural$-concave functions, it can be efficiently computed. To see this, we formulate the dual LP, and notice that separation corresponds to maximizing $f(S)-z(S)$.
 \begin{equation}\label{eq:concave-ext-primal}
  \begin{aligned}
      &\min& z^{\top}x + \alpha \qquad \quad & \\
      &\text{s.t.}&  z(T) + \alpha \geq f(S) & \qquad \forall S\subseteq E\, .
  \end{aligned}
\end{equation}
\begin{proposition}\label{prop:submod-max-m-concave}
If $f:\, 2^E\to\R\cup\{-\infty\}$ is an $M^\natural$-concave function, then for every point $x\in [0,1]^E$, $\hat f(x)$ can be computed in time polynomial in $|E|$ and $\mu(f)$.
\end{proposition}
We note that the existence of a concave extension satisfying desirable combinatorial properties is equivalent to the function being $M^\natural$-concave, see \cite[Theorem 6.43]{Murotabook}.

\paragraph{Probability distributions}
Let $\Bin(n,p)$ denote the binomial distribution with parameters $n$ and $p$, and let $\Poi(\lambda)$ denote the Poisson distribution with parameter $\lambda$. Recall that $\Pr(\Poi(\lambda)= k)=e^{-\lambda} \lambda^k/k!$ for any $k\in \Z_+$.

\begin{definition}
Given random variables $X$ and $Y$, we say that \emph{$X$ is at least $Y$ in the concave order} if for every concave function $\varphi:\R\rightarrow \R$, we have $\E[\varphi(X)]\geq \E[\varphi(Y)]$ whenever the expectations exist. It is denoted as $X\geq_{\text{cv}}Y$.
\end{definition}

In particular, we will use the following relation between the binomial and Poisson distributions:

\begin{lemma}[{\cite[Lemma 2.1]{journals/mp/BarmanFGG22}}]\label{lem:concave-order}
For any $n\in \N$ and $p\in [0,1]$, we have $\Bin(n,p) \geq_{\text{cv}} \Poi(np)$. 
\end{lemma}

\subsection{Properties of the Multilinear Extension}\label{sec:multilinear}
Let $f:\, 2^E\to\R_+$ be an arbitrary set function, and $F:\, [0,1]^E\to \R_+$ be its multilinear extension.
We will use the following well known properties, see e.g.~\cite{journals/siamcomp/CalinescuCPV11}.

\begin{proposition}\label{prop:multilinear}
For any $x\in [0,1]^E$ and $i\in E$, the function $\phi(t):=F(x+te_i)$ is linear.
\end{proposition}

\begin{proposition}\label{prop:multilinear-monotone}
If $f$ is monotone, then $F(x)\geq F(y)$ for all $x\geq y$.
\end{proposition}

\begin{proposition}\label{prop:multilinear-convex}
If $f$ is submodular, then for any $x\in [0,1]^E$ and $i,j\in E$, the function $\phi(t) := F(x + t(e_i-e_j))$ is convex.
\end{proposition}

\begin{proposition}\label{prop:multilinear-gradient}
For any $y\in [0,1]^E$, the gradient of $F$ at $y$ is given by
\[\nabla F(y)_i = \frac{\partial F}{\partial x_i}(y) = \E[f(Y+i)] - \E[f(Y)],\]
where $Y$ is the random set obtained by selecting each element $j\in E\setminus \{i\}$ independently with probability $y_j$.
Consequently, $\nabla F(y)\geq \0$ if $f$ is nondecreasing.
\end{proposition}

\begin{proposition}\label{prop:multilinear-hessian}
For any $z\in [0,1]^E$, the Hessian of $F$ at $z$ is given by
\[H(z)_{ij} = \frac{\partial^2 F}{\partial x_i \partial x_j}(z) = \begin{cases}
  \E[f(Z+i+j)] - 
\E[f(Z+i)] - \E[f(Z+j)] + \E[f(Z)] &\text{ if }i\neq j,\\
  0 &\text{ if }i=j,
\end{cases}\]
where $Z$ is the random set obtained by selecting each element $k\in E\setminus \{i,j\}$ independently with probability $z_k$. Consequently, $H(z)\leq \0$ if $f$ is submodular.
\end{proposition}

\section{Correlation Gap Bounds for Monotone Submodular Maximization}\label{sec:max-corr}

In this section, we focus on the special case of \eqref{prob:submod-max} considered by Shioura~\cite{journals/dmaa/Shioura09}, in which $f=\sum_{j=1}^m f_j$ where every $f_j:2^E\to \R$ is a monotone $M^\natural$-concave function.
As mentioned in the introduction, a more specialized version of this problem was considered by Calinescu et al.~\cite{conf/ipco/CalinescuCPV07}, in which every $f_j$ is a weighted matroid rank function.
We first show that the pipage rounding algorithm of Shioura~\cite{journals/dmaa/Shioura09} and Calinescu et al.~\cite{conf/ipco/CalinescuCPV07} achieves an approximation ratio of $\min_{i\in [m]}\CG(f_j)\geq 1-1/e$ (Proposition~\ref{prop:sum-mconcave}).
Then, we demonstrate how the approximation results for various concave coverage problems \cite{conf/stacs/BarmanFF21,journals/mp/BarmanFGG22,Dudycz2020} are implied by Proposition~\ref{prop:sum-mconcave} by analyzing the correlation gap of the constituent function $f_j$'s.

\subsection{Proof of Proposition~\ref{prop:sum-mconcave}}\label{sec:mnat-sum}

Let us define $ \tilde f:\, [0,1]^E\to \R_+$ as the sum of the concave extensions.
 \begin{equation}\label{eq:tilde-f}
 \tilde f(x):=\sum_{j=1}^m \hat f_j(x).
 \end{equation}
Note that $\tilde f(x)\le \hat f(x)$; however, this inequality may be strict.
Shioura's algorithm~\cite{journals/dmaa/Shioura09} starts by solving
 \begin{equation}\label{eq:max-tilde-f}
 \max_{x\in \cP} \tilde f(x)
 \end{equation}
 This is a convex optimization problem, and is also a relaxation of \eqref{prob:submod-max}, noting that 
 for any $S\subseteq E$, $\tilde f(\chi_S)=f(S)$.

  The number of constraints in $\cP$ is exponential, but can be efficiently separated over. The objective function $\tilde f(x)$ can be evaluated by solving $m$ exponential-size linear programs.
Shioura showed that \eqref{eq:max-tilde-f} can be solved using the ellipsoid method by implementing a subgradient oracle. The algorithm returns an exact solution in time polynomial in $n$, $m$, and the complexity parameter $\lowerp(f)$, assuming the functions are rational-valued.

Given an optimal solution $x^*$ to \eqref{eq:max-tilde-f}, the pipage rounding technique first introduced by  Ageev and Sviridenko~\cite{journals/jco/AgeevS04} can be used to obtain a set $S\in \cI$ with $f(S)\ge F(x^*)$. Hence, we obtain an $\alpha$-approximation for \eqref{prob:submod-max} as long as we can show $F(x^*)\ge \alpha \tilde f(x^*)$. The proof of Proposition~\ref{prop:sum-mconcave} is complete by the following lemma.
\begin{lemma}
Let $\alpha \coloneqq \min_{j=1}^m \CG(f_j)$. Then,
for every $x\in [0,1]^E$, $F(x)\ge \alpha \tilde f(x)$.
\end{lemma}
\begin{proof}
Let $F_j$ be the multilinear extension of $f_j$. Note that $F(x)=\sum_{j=1}^m F_j(x)$.
By the definition of the correlation gap,
\[
 F(x) = \sum_{j=1}^m F_j(x)\ge \alpha \sum_{j=1}^m \hat f_j(x)=\alpha\tilde f(x)\, . \qedhere
 \]
\end{proof}

\subsection{Concave Coverage Problems}\label{sec:barman}
We now discuss the concave coverage  model in Barman et al.~\cite{conf/stacs/BarmanFF21}, and show that the Poisson curvature studied in this paper can be interpreted as a correlation gap bound. Further, in Proposition~\ref{prop:uniform-bound}, we show that the tight bounds in \cite{journals/mp/BarmanFGG22} for the maximum multicoverage problems coincide with the correlation gap bound in Theorem~\ref{thm:monster} for uniform matroids, i.e., $\mgir=\mrk+1$.

Let $\mathcal{M} = (E,\mathcal{J})$ be a matroid, and let $\varphi:\Z_+\rightarrow \R_+$ be a normalized nondecreasing concave function, i.e., $\varphi(0) = 0$, $\varphi(1) = 1$, $\varphi(i+1)\geq \varphi(i)$ and $\varphi(i+1) - \varphi(i) \geq \varphi(i+2) - \varphi(i+1)$ for all $i\in \Z_+$.
For every $j\in[m]$, we are given a subset $E_j\subseteq E$, a weight $w_j\in \R_+$, and a function $f_j:2^E\rightarrow \R_+$ defined by $f_j(S) := \varphi(|S\cap E_j|)$.
In the \emph{$\varphi$-MaxCoverage} problem, the goal is to maximize $f(S):=\sum_{j=1}^m w_j f_j(S)$ subject to $S\in \mathcal{J}$.
Barman et al.~\cite{conf/stacs/BarmanFF21} gave an approximation algorithm for this problem, whose approximation factor is the so-called \emph{Poisson concavity ratio} of $\varphi$, defined as
\[\alpha_{\varphi} := \inf_{\lambda\in \R_+}\frac{\E[\varphi(\Poi(\lambda))]}{\hat{\varphi}(\E[\Poi(\lambda)])} = \inf_{\lambda\in \R_+}\frac{\E[\varphi(\Poi(\lambda))]}{\hat{\varphi}(\lambda)}.\]
Here, $\hat{\varphi}:\R_+\rightarrow \R_+$ is the concave extension of $\varphi$, i.e.~$\hat{\varphi}(\lambda) = \varphi(\floor{\lambda}) + (\varphi(\floor{\lambda}+1) - \varphi(\floor{\lambda}))(\lambda -\floor{\lambda})$.

In this subsection, we show that the correlation gap of each $f_j$ is at least the Poisson concavity ratio of $\varphi$.
To this end, fix a $j\in [m]$.
The following lemma relates the concave extensions of $f_j$ and $\varphi$; the proof of this and the next lemma are given in the Appendix.

\begin{restatable}{lemma}{concaveextoned}\label{lem:concave-ext-1d}
For any $x\in [0,1]^E$, we have $\hat{f}_j(x) = \hat{\varphi}(x(E_j))$ 
\end{restatable}

The next lemma shows that the multilinear extension $F_j$ is minimized at `symmetric' points.

\begin{restatable}{lemma}{symmetric}\label{lem:symmetric}
For any $x\in [0,1]^E$, let $\bar{x}\in [0,1]^E$ be the vector given by
\[\bar{x}_i:=\begin{cases}
    \frac{x(E_j)}{|E_j|}, &\text{ if }i\in E_j\\
    x_i, &\text{ otherwise.}
\end{cases}.\]
Then, $F_j(x)\geq F_j(\bar{x})$.
\end{restatable}
We show that the Poisson concavity ratio is a lower bound on the correlation gap:
\begin{proposition}\label{prop:poi-curve}
We have $\CG(f_j)\geq \alpha_{\varphi}$. %
\end{proposition}

\begin{proof}
Let $n_j = |E_j|$.
For any point $x\in [0,1]^E$, let $\lambda = x(E_j)$.
Define the vector $\bar{x}\in [0,1]^E$ as $\bar{x}_i := \lambda/n_j$ if $i\in E_j$, and $\bar{x}_i := x_i$ otherwise.
According to Lemmas \ref{lem:concave-order} and \ref{lem:symmetric},
\begin{equation}\label{eq:poi-curve}
    F_j(x)\geq F_j(\bar{x}) = \sum_{k=0}^{n_j} \varphi(k){n_j \choose k}\pr{\frac{\lambda}{n_j}}^k\pr{1-\frac{\lambda}{n_j}}^{n_j-k} = \E\br{\varphi\pr{\Bin\pr{n_j,\frac{\lambda}{n_j}}}}\geq \E\br{\varphi(\Poi(\lambda))}.
\end{equation}
Moreover, we have $\hat{f}_j(x) = \hat{\varphi}(\lambda)$ by Lemma \ref{lem:concave-ext-1d}.
Hence, $F_j(x)/\hat{f}_j(x) \geq \E\br{\varphi(\Poi(\lambda))}/\hat{\varphi}(\lambda)$.
\end{proof}

We remark that the inequality in Proposition \ref{prop:poi-curve} is asymptotically tight. For any $\lambda\geq 0$, if we choose $x = (\lambda/n_j)\cdot\1$, then the first inequality in \eqref{eq:poi-curve} is tight. On the other hand, the second inequality in \eqref{eq:poi-curve} is asymptotically tight as $n_j\to \infty$.

When $f_j$ is the rank function of a rank-$\ell$ uniform matroid, \cite{journals/mp/BarmanFGG22} gave a tight approximation ratio  $1-\frac{e^{-\rgir}\rgir^{\rgir}}{\rgir!}$. We show that this coincides with the lower bound in Theorem~\ref{thm:monster}. The proof is given in the Appendix.

\begin{restatable}{proposition}{uniformbound}\label{prop:uniform-bound}
For every $\ell\in \N$, we have
\[1-\frac{1}{e} + \frac{e^{-\rgir}}{\rgir} \pr{\sum_{i=0}^{\rgir-1} (\rgir-i) \br{{\rgir \choose i}(e-1)^i - \frac{\rgir^i}{i!}}} = 1-\frac{e^{-\rgir}\rgir^{\rgir}}{\rgir!}\, .\]
\end{restatable}

\section{Locating the Correlation Gap}
\label{sec:locating}
In this section, given a weighted matroid rank function $r_w$, we locate a point $x^*\in [0,1]^E$ on which the correlation gap $\CG(r_w)$ is realized, and derive some structural properties. 
Using this, we prove Theorem~\ref{thm:weighted}, which states that the smallest correlation gap over all possible weightings is attained by uniform weights. 

We remark that the existence of $x^*$ is a priori not clear as the correlation gap is defined using an infimum.
In Appendix~\ref{sec:attainment}, we prove that the correlation gap is always attained for a nonnegative monotone submodular function (Theorem~\ref{thm:attain}).
Interestingly, neither the monotonicity nor submodularity assumption can be dropped.
We provide examples of such functions in Appendix~\ref{sec:attainment}.

It will be helpful to work with a more convenient characterization of the concave extension of $r_{w}$. 
Recall the definition of $\hat r_{w}$ in \eqref{eq:concave-extension} and its dual form~\eqref{eq:concave-ext-primal}. We first show that the equalities in \eqref{eq:concave-extension} can be relaxed to inequalities for any monotone submodular function.

\begin{lemma}\label{lem:concave-extension-relaxed}
For any monotone submodular function $f:\, 2^E\to\R$ and $x\in[0,1]^E$, its concave extension $\hat{f}(x)$ can be equivalently written as
\begin{equation}\label{eq:concave-relaxed-pd}
\max\left\{\sum_{S\subseteq E} \lambda_S f(S):\, \sum_{S\subseteq E: i\in S} {\lambda_S}\le x_i\, \,\forall i\in E\, ,\, 
 \sum_{S\subseteq E} {\lambda_S}=1\, ,\,  \lambda\ge 0\right\}\, .
\end{equation}
\end{lemma}
\begin{proof}
Clearly, the optimal value of \eqref{eq:concave-relaxed-pd} is at least $\hat f(x)$.
Take an optimal solution $\lambda$ to \eqref{eq:concave-relaxed-pd} such that $\delta(\lambda):=\sum_{i\in E} \left(x_i-\sum_{S\subseteq E: i\in S} {\lambda_S}\right)$ is minimal. If $\delta=0$, then $\lambda$ is also feasible to \eqref{eq:concave-extension}, proving the claim. Assume that $\delta>0$, and take any $i\in E$ for which $x_i>\sum_{S\subseteq E: i\in S} {\lambda_S}$. Since $x_i\le 1$ and $\sum_{S\subseteq E} {\lambda_S}=1$, there exists a set $T\subseteq E$ with $\lambda_T>0$ and $i\notin T$.

Let us modify this solution to $\lambda'$ defined as  
$\lambda'_{T + i}=\lambda_{T+i}+\varepsilon$, 
$\lambda'_T=\lambda_T-\varepsilon$, and $\lambda'_S=\lambda_S$ otherwise. 
For small enough $\varepsilon>0$, $\lambda'$ is also a feasible solution to \eqref{eq:concave-relaxed-pd} with $\delta(\lambda')<\delta(\lambda)$. 
Moreover, $\lambda'$ is also optimal, since $\sum_{S\subseteq E} \lambda'_S f(S)\ge \sum_{S\subseteq E} \lambda_S f(S)$ by the monotonicity of $f$. 
This contradicts the choice of $\lambda$; consequently, $\delta(\lambda)=0$ must hold and the claim follows.
\end{proof}

\begin{lemma}\label{lem:concave-ext-char}
Let $\cM=(E,\cI)$ be a matroid with rank function $\rk$ and weights $w\in \R_+^E$. 
For any $x\in[0,1]^E$, 
\[
\hat \rk_{w}(x)=\max\{w^\top y:\, y\in\pind(r),\, y\le x\}\, .
\]
\end{lemma}
\begin{proof}
Consider an optimal solution $\lambda$ to the LP in \eqref{eq:concave-relaxed-pd} for $f=\rk_w$ with $\sum_{S\subseteq E} \lambda_S|S|$ minimal.
We claim that every $S\subseteq E$ with $\lambda_S>0$ must be independent. 
Indeed, recall that $r_w(S)=w(T)$ for some independent set $T\subseteq S$. If $S\notin\cI$, then we can simply replace $S$ in the combination by this set $T$. The solution remains feasible with the same objective value, but smaller $\sum_{S\subseteq E} \lambda_S|S|$. 

Consequently, we may assume that $r_w(S)=w(S)$ for every $S\in\supp(\lambda)$. Letting $y_i=\sum_{S\subseteq E:\, i\in S}\lambda_i$, the objective of  \eqref{eq:concave-relaxed-pd} can be written as
\[
\sum_{S\subseteq E}\lambda_S r_w(S)=w^\top y\, .
\]
Note that $y\le x$ and $y\in\pind(r)$, since $y$ can be written as a convex combination of incidence vectors of independent sets. 
Hence, \eqref{eq:concave-relaxed-pd} for $f=\rk_w$ is equivalent to maximizing $w^\top y$ over $y\in\pind(r)$, $y\le x$, proving the statement.
\end{proof}

Next, we show that there exists a point $x^*$ in the independent set polytope $\mathcal{P}(r)$ on which the correlation gap $\CG(r_w)$ is realized. Furthermore, such a point $x^*$ can be chosen such that $\supp(x^*)$ is a \emph{tight set} with respect to $x^*$, i.e., $x^*(E) = r(\supp(x^*))$.

\begin{theorem} \label{thm:locating}
Let $\cM=(E,\cI)$ be a matroid with rank function $r$. For any weights $w\in\R_+^E$, there exists a point $x^*\in \pind(\rk)$ such that $x^*(E)=\rk(\supp(x^*))$ and
\[
\CG(r_w)=\frac{\Rk_w(x^*)}{\hat \rk_w(x^*)}=\frac{\Rk_w(x^*)}{w^\top x^*}\, .
\]
\end{theorem}

\begin{proof}
By Theorem~\ref{thm:attain}, a minimizer of ${\Rk_w(x)}/{\hat \rk_w(x)}$ in $[0,1]^E$ exists.  
First, we prove that it can be found in the independent set polytope $\pind(r)$. 
Take a minimizer $x\notin\pind(r)$. 
By Lemma~\ref{lem:concave-ext-char}, $\hat \rk_w(x)=w^\top y$ for some $y\in \pind(r)$, $y\le x$. 
Clearly, $\hat \rk_w(y)=w^\top y$. By Proposition \ref{prop:multilinear-monotone}, we have $\Rk_w(y)\le \Rk_w(x)$. 
This proves that ${\Rk_w(x)}/{\hat \rk_w(x)}\ge {\Rk_w(y)}/{\hat \rk_w(y)}$, thus, equality must hold and $y$ is also a minimizer of the ratio.

For the rest of the proof, consider a minimizer $x\in\pind(r)$. Note that ${\Rk_w(x)}/{\hat \rk_w(x)}=\Rk_w(x)/w^\top x$. Among such minimizers, let us pick $x$ such that $\supp(x)$ is minimal. The proof is complete by showing that  $x(S)=\rk(S)$ for $S\coloneqq\supp(x)$. 

For a contradiction, assume $x(S)<\rk(S)$.
We claim that there exists a $j\in S$ such that $x+\varepsilon\chi_j\in \pind(r)$ for some $\varepsilon>0$. If no such $j$ exists, then there exists a set $T_j$ for each $j\in S$ such that $j\in T_j$, and $x(T_j)=\rk(T_j)$. 
For any $j,k\in S$, we can uncross $T_j$ and $T_k$
\[x(T_j\cup T_k) + x(T_j\cap T_k) = x(T_j) + x(T_k) = r(T_j) + r(T_k) \geq r(T_j\cup T_k) + r(T_j\cap T_k)\]
to deduce that $x(T_j\cup T_k) = r(T_j \cup T_k)$ and $x(T_j\cap T_k) = r(T_j\cap T_k)$.
Repeating this operation yields $x(T)=\rk(T)$ for $T=\cup_{j\in S} T_j$.
Clearly, $S\subseteq T$. But this implies $x(S)=\rk(S)$ since $x(S)=x(T)$ and $\rk(S)\le \rk(T)$.
Thus, there exists a $j\in S$ such that $x+\varepsilon\chi_j\in \pind(r)$ for some $\varepsilon>0$.

For $\gamma\in[0,1]$, let $x^\gamma$ be the vector obtained from $x$ by replacing $x_j$ with $\gamma$. Let $\Gamma=\max\{\gamma:\, x^\gamma\in \pind(r)\}$. By the choice of $j$, $x_j<\Gamma$.

According to Proposition~\ref{prop:multilinear}, $h(\gamma)\coloneqq\Rk_w(x^\gamma)$ is a linear function in $\gamma$;  we can write $h(\gamma)=a+b\gamma$ for $a,b\in\R_+$.
For $\gamma\in [0,\Gamma]$, $x^\gamma\in\pind(r)$, and therefore $\hat \rk(x^\gamma)=w^{\top}x^\gamma$; 
this is also a linear function and can be written as $\hat r(x^\gamma)=c+d\gamma$, where $c=\sum_{i\neq j} w_i x_i$ and $d=w_j$. Hence, for $\gamma\in [0,\Gamma]$, we can write
\[
\frac{\Rk_w(x^\gamma)}{\hat \rk_w(x^\gamma)}=\frac{a+b\gamma}{c+d\gamma}\, .
\]
It is easy to see that if $a/c<b/d$, then the unique minimizer of this ratio is $\gamma=0$; if $a/c>b/d$, then the unique minimizer is $\gamma=\Gamma$. Both these cases contradict the optimal choice of $x$. Hence, we must have $a/c=b/d$, in which case the ratio is constant on $\gamma\in [0,\Gamma]$. Therefore, $x^0$ is also a minimizer. This is a contradiction to the minimal choice of $\supp(x)$.
\end{proof}

We are ready to prove Theorem~\ref{thm:weighted}.

\begin{proof}[Proof of Theorem~\ref{thm:weighted}]
For a contradiction, assume there exists a weight vector $w\geq 0$ and a point $x^*\in [0,1]^E$ such that $R_w(x^*)/\hat \rk_w(x^*)<\CG(\rk)$. According to Theorem~\ref{thm:locating}, we can assume $x^*\in \pind(\rk)$ and thus $\hat\rk_w(x^*)=w^\top x^*$.

Let $w^1>w^2>\dots>w^k\geq 0$ denote the distinct values of $w$. %
For each $i\in [k]$, let $E_i\subseteq E$ denote the set of elements with weight $w_i$. 
Clearly, $k\ge 2$ as otherwise 
$R_w(x^*)/\hat \rk_w(x^*)=w^1R(x^*)/w^1 x^*(E)=R(x^*)/x^*(E)\ge \CG(\rk)$. Let us pick a counterexample with $k$ minimal.

First, we claim that $x^*(E_i)>0$ for all $i\in [k]$.
Indeed, if $x^*(E_i) = 0$, then $x^*_e = 0$ for all $e\in E_i$.
So, for every $e\in E_i$, we can modify $w_e \gets w^j$ where $j\neq i$ without changing the value of $R_w(x^*)/\hat \rk_w(x^*)$. However, this contradicts the minimal choice of $k$.
By the same argument, we also have $w^k>0$.

Let $X$ be a random set obtained by sampling every element $e\in E$ independently with probability $x^*_e$.
Let $I_X\subseteq X$ denote a maximum weight independent subset of $X$. Recall the well-known  property of matroids that a maximum weight independent set can be selected greedily in decreasing order of the weights $w_e$. We fix an arbitrary tie-breaking rule inside each set $E_i$.

The correlation gap of $\rk_w$ is given by
\[\frac{R_w(x^*)}{\hat{\rk}_w(x^*)} = \frac{\sum_{S\subseteq E}\Pr(X=S)\rk_w(S)}{w^{\top}x^*} = \frac{\sum_{e\in E}w_e\Pr(e\in I_X)}{\sum_{e\in E}w_e x^*_e} = \frac{\sum_{i=1}^k w^i \sum_{e\in E_i} \Pr(e\in I_X)}{\sum_{i=1}^k w^i x^*(E_i)}.\]
 Consider the set 
\[J:=\argmin_{i\in [k]} \frac{w^i\sum_{e\in E_i}\Pr(e\in I_X)}{w^ix^*(E_i)}.\]
We claim that $J\setminus \{1\}\neq \emptyset$.
Suppose that $J = \{1\}$ for a contradiction.
Define the point $x'\in\pind(\rk)$ as $x'_e := x^*_e$ if $e\in E_1$, and $x'_e := 0$ otherwise.
Then, we get a contradiction from
\[\CG(r)\le \frac{R(x')}{\hat{\rk}(x')} = \frac{w^1\sum_{e\in E_1}\Pr(e\in I_X)}{w^1 x^*(E_1)} < \frac{\sum_{i=1}^k w^i \sum_{e\in E_i} \Pr(e\in I_X)}{\sum_{i=1}^k w^i x^*(E_i)} = \frac{R_w(x^*)}{\hat{\rk}_w(x^*)}.\]
The first equality holds because for each element $e\in E_1$, $\Pr(e\in I_X)$ only depends on $x^*_{E_1} = x'_{E_1}$. This is by the greedy choice of $I_X$:  elements in $E_1$ are selected regardless of $X\setminus E_1$.
The strict inequality is due to $J = \{1\}$, $k\geq 2$ and $x^*(E_i),w^i>0$ for all $i\in [k]$.

Now, pick any index $j\in J\setminus \{1\}$.
Then,
\[\frac{w^j\sum_{e\in E_j}\Pr(e\in I_X)}{w^jx^*(E_j)} \leq \frac{\sum_{i\neq j} w^i \sum_{e\in E_i} \Pr(e\in I_X)}{\sum_{i\neq j} w^i x^*(E_i)}.\]
So, we can increase $w^j$ to $w^{j-1}$ without increasing the correlation gap. That is, defining $\bar{w}\in \R_+^E$ as $\bar{w}_e := w^{j-1}$ if $e\in E_j$ and $\bar{w}_e := w_e$ otherwise, we get
\begin{align*}
  \frac{R_w(x^*)}{\hat{\rk}_w(x^*)} &\geq \frac{w^{j-1}\sum_{e\in E_j}\Pr(e\in I_X) + \sum_{i\neq j}w^i \sum_{e\in E_i}\Pr(e\in I_X)}{w^{j-1}x^*(E_j) + \sum_{i\neq j} w^i x^*(E_i)} \\
  &= \frac{\sum_{e\in E}\bar{w}_e \Pr(e\in I_X)}{\sum_{e\in E}\bar{w}_ex^*_e} = \frac{\sum_{S\subseteq E}\Pr(X=S)r_{\bar{w}}(S)}{\bar{w}^{\top}x^*}= \frac{R_{\bar{w}}(x^*)}{\hat{r}_{\bar{w}}(x^*)} \enspace .
\end{align*}
The second equality holds because for every $S\subseteq E$, $I_S$ remains a maximum-weight independent set with the new weights $\bar{w}$. This again contradicts the minimal choice of $k$.
\end{proof}

\section{The Correlation Gap Bound for Matroids}\label{sec:gap-analysis}
This section is dedicated to the proof of Theorem~\ref{thm:monster}. 
For the matroid $\mathcal{M}=(E,\mathcal{I})$, let $\rk$ denote the rank function,  $\mrk=\rk(E)$ the rank, and $\mgir$ the girth. We have $\mgir>1$ since the matroid is assumed to be loopless.

According to Theorem~\ref{thm:locating}, there exists a point $x^*\in \pind(\rk)$ and a set $S\subseteq E$ such that $x^*(S) = \rk(S)$, $x^*(E\setminus S) = 0$, and $\CG(\rk) = \Rk(x^*)/\rk(x^*)$. 
For notational convenience, let us define 
\[
\rgir \coloneqq \mgir-1 \,, \qquad \qquad \qquad \xrk \coloneqq x^*(E) = x^*(S) = \rk(S) \enspace .
\] %
Note that if $\xrk< \rgir$, then $S$ is independent.
As $x^*(S) = \rk(S) = |S|$ and $x^*\leq \1$, we have $x^*_i = 1$ for all $i\in S$.
In this case, it follows that $x^*$ is integral and $\Rk(x^*) = \hat{\rk}(x^*)$, so the correlation gap is 1.
Henceforth, we will assume that $\xrk\geq \rgir$.

\medskip

In this section, we analyze the multilinear extension of $\rk$.
Let $g \colon 2^E \to \Z_+$ be the rank function of a uniform matroid of rank $\rgir$ over ground set $E$, and define the function $h := \rk - g$.
Clearly, $\rk = g + h$.
By linearity of expectation, the multilinear extension of $\rk$ can be written as
\begin{equation}
\Rk(x) = \E[r(S)] = \E[g(S)+h(S)] = \E[g(S)] + \E[h(S)] = G(x) + H(x) \enspace ,
\end{equation}
where $G$ and $H$ are the multilinear extensions of $g$ and $h$ respectively.
To lower bound $\Rk(x^*)$, we will lower bound $G(x^*)$ and $H(x^*)$ separately.

\subsection{Lower Bounding \texorpdfstring{$G(x^*)$}{G(x*)}}

Observe that $G$ is a symmetric polynomial because $g$ is the rank function of a uniform matroid.
As $g$ is submodular, Proposition~\ref{prop:multilinear-convex} indicates that $G$ is convex along $e_i -e_j$ for all $i,j\in E$. 
The next lemma is an easy consequence of these two properties. We have already proven it in a more general form in Lemma~\ref{lem:symmetric}. 

\begin{lemma}\label{lem:symmetric-2}
For any $x\in [0,1]^E$, we have $G(x)\geq G((x(E)/n)\cdot\1)$.
\end{lemma}

By Lemma \ref{lem:symmetric-2}, we have
\begin{equation}\label{eq:symmetric}
  G(x^*)\geq G\pr{\frac{\xrk}{n}\cdot\1} = \sum_{k=0}^n \min\{k,\rgir\}{n \choose k}\pr{\frac{\xrk}{n}}^k\pr{1-\frac{\xrk}{n}}^{n-k} = \E\br{\min\set{\Bin\pr{n,\frac{\xrk}{n}},\rgir}} \enspace .
\end{equation}
In other words, we can lower bound $G(x^*)$ by the expected value of $\Bin(n,\xrk/n)$ truncated at $\rgir$. We now use Lemma~\ref{lem:concave-order} on the concave order of the binomial and Poisson distributions to obtain
\begin{equation} \label{eq:Bin-Poi-concave-order}
  \E\br{\min\set{\Bin\pr{n,\frac{\xrk}{n}},\rgir}} \geq \E\br{\min\set{\Poi(\xrk),\rgir}} = \sum_{k=0}^\infty \min\{k,\rgir\} \frac{\xrk^ke^{-\xrk}}{k!} \enspace .
\end{equation}
Using $\Pr(\Poi(\xrk)\geq j) = \sum_{k=j}^\infty \frac{\xrk^ke^{-\xrk}}{k!}$, this amounts to  
\begin{equation} \label{eq:Poisson-series-estimate}
  G(x^*)\geq \sum_{j=1}^\rgir \Pr(\Poi(\xrk)\geq j) = \sum_{j=1}^\rgir \pr{1-\sum_{k=0}^{j-1} \frac{\xrk^ke^{-\xrk}}{k!}} = \rgir - \sum_{k=0}^{\rgir-1}(\rgir-k)\frac{\xrk^ke^{-\xrk}}{k!} \enspace .
\end{equation}

\subsection{Lower Bounding \texorpdfstring{$H(x^*)$}{H(x*)}}

Next, we turn to the extension $H$.
We first describe the general setup, which is to incrementally build a set $\QS(1)$ as follows.
For each element $i\in E$, we put a Poisson clock of rate $x^*_i$ on it.
We initialize with $\QS(0)=\emptyset$, and start all the clocks simultaneously at time $t=0$.
For $t\in [0,1]$, if the clock on an element rings at time $t$, we add that element to our current set.
This process is terminated at time $t=1$.
This gives rise to the time-dependent set-valued random variable $\QS$ such that, for $t\in [0,1]$, $\QS(t)$ is the random variable for the set at time $t$.
This process can also be viewed as a continuous-time Markov chain, where the state space is the power set $2^E$.
From a set/state $S$, the possible transitions are to those sets $S'$ where $S\subset S'$ and $|S'| = |S| + 1$. 
Note that the Markov property is satisfied because both the holding time and transitions only depend on the current state $\QS(t)$.

Due to the independence of the Poisson clocks, for every set $S\subseteq E$, we have
\[
\Pr[\QS(1)=S] = \prod_{i\in S} \Pr[i\in \QS(1)] \prod_{i\notin S}\Pr[i\notin \QS(1)] = \prod_{i\in S} (1-e^{-x^*_i}) \prod_{i\notin S} e^{-x^*_i} \enspace .
\]
Since $h$ is monotone and $x^*_i\geq 1-e^{-x^*_i}$ for all $i\in E$, Proposition \ref{prop:multilinear-monotone} gives
\begin{equation} \label{eq:lower-bound-Poisson-process}
H(x^*) \geq H(1-e^{-x^*}) = \E[h(\QS(1))] \enspace .
\end{equation}
So, it suffices to lower bound $\E[h(\QS(1))]$.

Let $t\in [0,1)$ and consider an infinitesimally small interval $[t,t+dt]$.
For every element $i\in E$, the number of times its clock rings is a Poisson random variable with rate $x^*_idt$.
Hence, the probability that an element $i$ is added to our set during this interval is 
\[
\Pr(\Poi(x^*_idt)\geq 1) = 1-e^{-x^*_idt} = 1 - (1 - x^*_idt + O(dt^2)) = x^*_idt + O(dt^2) \enspace ,
\]
where the second equality follows from Taylor's theorem.
Observe that the probability of adding two or more elements during this interval is also $O(dt^2)$.
Since $dt$ is very small, we can effectively neglect all $O(dt^2)$ terms.
Conditioning on the event $\QS(t) = S$, the expected increase of $h(\QS(t))$ (up to $O(dt^2)$ terms) is
\[\E[h(\QS(t+dt)) - h(\QS(t))|\QS(t)= S]= \sum_{i\in E}h_S(i)x^*_i dt.\]

From the definition of $h$, for each element $i\in E$, we have $h_S(i) = r_S(i)$ if $|S| \geq \ell$, and $h_S(i) = 0$ otherwise.
This motivates the following definition.

\begin{definition}
We say that $\QS$ is \emph{activated at time $t'$} if $|\QS(t)|<\rgir$ for all $t< t'$ and $|\QS(t)|\geq \rgir$ for all $t\geq t'$.
We call $t'$ the \emph{activation time} of $\QS$.

We denote the random variable for the activation time of $\QS$ by $T$. 
\end{definition}

For a fixed $t' \in \R_+$, if we further condition on the event $T=t'$, the expected increase of $h(\QS(t))$ (up to $O(dt^2)$ terms) is
\begin{equation}\label{eq:marginal_gain}
  \E[h(\QS(t+dt)) - h(\QS(t))|\QS(t)= S \land T=t'] = \sum_{i\in E}\rk_S(i)x^*_i dt
\end{equation}
for all $t\geq t'$ and $S\subseteq E$ where $|S|\geq \ell$. For such a set $S$, we have
\[h(S) + \sum_{i\in E}\rk_S(i)x^*_i = r(S) - \ell + \sum_{i\in E}\rk_S(i)x^*_i \geq r^*(x^*) - \rgir = \hat{\rk}(x^*) - \rgir  =\1^{\top}x^* - \rgir = \xrk - \rgir.\]
The inequality follows from the definition of $r^*$ in \eqref{eq:marginal}.
The second equality is by Theorem \ref{thm:box-intersect}, while the third equality is by Lemma \ref{lem:concave-ext-char} because $x^*\in \pind(\rk)$.
Hence, \eqref{eq:marginal_gain} becomes
\[\E[h(\QS(t+dt)) - h(\QS(t))|\QS(t)= S \land T=t'] \geq (\xrk - \rgir - h(S))dt\]
Dividing by $dt$ and taking expectation over $S$, we obtain for all $t\geq t'$,
\begin{equation}\label{eq:de}
  \frac{1}{dt}\E[h(\QS(t+dt)) - h(\QS(t))|T=t'] \geq \xrk - \rgir - \E[h(\QS(t))|T=t'].
\end{equation}
Let $\phi(t) := \E[h(\QS(t))|T=t']$. 
Then, \eqref{eq:de} can be written as $\frac{d\phi}{dt}\geq \xrk - \rgir - \phi(t)$.
To solve this differential inequality, let $\psi(t) := e^t\phi(t)$ and consider $\frac{d\psi}{dt} = e^t(\frac{d\phi}{dt} +  \phi(t)) \geq e^t(\xrk - \rgir)$.
Since $\psi(t') = \phi(t') = 0$, we get
\[\psi(t) = \int_{t'}^t \frac{d\psi}{ds}ds\geq \int_{t'}^t e^s(\xrk - \rgir)ds = (e^{t} - e^{t'})(\xrk - \rgir)\]
for all $t\geq t'$.
It follows that $\E[h(\QS(t))|T=t'] = \phi(t) = e^{-t}\psi(t) \geq (1-e^{t'-t})(\xrk - \rgir)$ for all $t\geq t'$.
In particular, at time $t=1$, we have $\E[h(\QS(1))|T=t'] \geq (1-e^{-(1-t')})(\xrk - \rgir)$ for all $t'\leq 1$.
Let $f_T$ denote the probability density function of $T$.
By the law of total expectation,
\begin{align}\label{eq:total_expectation}
  \E[h(\QS(1))] = \E_T[\E[h(\QS(1))|T= t]] &= \int_0^\infty f_T(t)\E[h(\QS(1))|T= t] dt \notag\\
  &= \int_0^1 f_T(t) \E[h(\QS(1))|T= t] dt \notag \\
  &\geq (\xrk - \rgir)\int_0^1 f_T(t)(1-e^{-(1-t)})dt.
\end{align}
    
Now, the cumulative distribution function of $T$ is given by
\begingroup
\allowdisplaybreaks
\begin{align*}
  \Pr(T\leq t) &= 1 - \sum_{\substack{S\subseteq E:\\|S|<\rgir}} \;\prod_{i\in S}(1-e^{-x^*_it})\prod_{i\notin S}e^{-x^*_it} \\
  &= 1- \sum_{\substack{S\subseteq E:\\|S|<\rgir}} \;\sum_{T\subseteq S} (-1)^{|T|} e^{-x^*(T\cup (E\setminus S))t} \\
  &= 1 - \sum_{S\subseteq E} \sum_{\substack{T\subseteq S:\\|T|+|E\setminus S|<\rgir}} (-1)^{|T|} e^{-x^*(S)t}  \tag{Change of variables $S\gets T\cup E\setminus S$}\\
  &= 1 - \sum_{S\subseteq E} \sum_{k=0}^{|S|+\rgir-n-1} (-1)^k {|S|\choose k} e^{-x^*(S)t} \tag{$|T| \leq \rgir - (n - |S|) -1$} \\
  &= 1-\sum_{S\subseteq E} (-1)^{|S|+\rgir-n-1}{|S|-1 \choose |S|+\rgir-n-1}e^{-x^*(S)t} \tag{Claim \ref{clm:binom-single}}\\
  &= 1-\sum_{S\subseteq E} (-1)^{|S|+\rgir-n-1}{|S|-1 \choose n-\rgir}e^{-x^*(S)t}.
\end{align*}
\endgroup
Differentiating with respect to $t$ yields the probability density function of $T$
\[f_T(t) = \frac{d}{dt}\Pr(T\leq t)= \sum_{S\subseteq E} (-1)^{|S|+\rgir-n-1}{|S|-1 \choose n-\rgir}x^*(S)e^{-x^*(S)t}.\]
Note that ${|S|-1\choose n-\rgir}>0$ if and only if $|S|\geq n+1-\rgir$.
Plugging this back into \eqref{eq:total_expectation} gives us

\begin{align} 
  \E[h(\QS(1))] &\geq (\xrk - \rgir)\sum_{S\subseteq E} (-1)^{|S|+\rgir-n-1}{|S|-1 \choose n-\rgir}x^*(S) \int_0^1 e^{-x^*(S)t} - e^{-1-(x^*(S)-1)t} dt \notag \\
  &= (\xrk - \rgir)\sum_{S\subseteq E} (-1)^{|S|+\rgir-n-1}{|S|-1 \choose n-\rgir}\left(1 - e^{-1} - \frac{e^{-1}-e^{-x^*(S)}}{x^*(S)-1}\right) \,, \label{eq:expectation} %
\end{align}
where the equality is due to %

\begin{align*}
  x^*(S) \int_0^1 e^{-x^*(S)t} - e^{-1-(x^*(S)-1)t} dt &= x^*(S) \br{-\frac{e^{-x^*(S)t}}{x^*(S)} + \frac{e^{-1-(x^*(S)-1)t}}{x^*(S)-1}}_0^1 \\
  &= \br{-e^{-x^*(S)t} + \pr{1+\frac{1}{x^*(S)-1}} e^{-1-(x^*(S)-1)t}}_0^1 \\
  &= -e^{-x^*(S)} + 1 + \pr{1+\frac{1}{x^*(S)-1}}\pr{e^{-x^*(S)} -e^{-1}} \\
  &= 1 - e^{-1} - \frac{e^{-1}-e^{-x^*(S)}}{x^*(S)-1}.
\end{align*}

Observe that \eqref{eq:expectation} is well-defined because whenever $x^*(S)=1$, L'H\^{o}pital's rule gives us
\[\frac{e^{-1}-e^{-x^*(S)}}{x^*(S)-1} = \lim_{t\rightarrow 1} \frac{e^{-1} - e^{-t}}{t-1} = \lim_{t\rightarrow 1} \frac{e^{-t}}{1} = e^{-1}.\]

Since $\rgir>0$ (as $\mathcal{M}$ has no loops), Claim~\ref{clm:binomial} allows us to extract the first part of \eqref{eq:expectation} as
\[
\sum_{S\subseteq E} (-1)^{|S|+\rgir-n-1}{|S|-1 \choose n-\rgir}(1 - e^{-1}) =  (1 - e^{-1})\sum_{k = 0}^{n}(-1)^{k+\rgir-n-1}{n \choose k}{k-1 \choose n-\rgir} \enspace = 1-e^{-1} .
\]
Pulling out a factor $-1$ from the remaining term, \eqref{eq:expectation} becomes
\begin{equation}\label{eq:expectation1}
\E[h(Q(1)] \geq (\xrk - \rgir)\br{1-e^{-1} + e^{-1}\sum_{S\subseteq E}(-1)^{|S|+\rgir-n}{|S|-1 \choose n-\rgir}\frac{1-e^{-(x^*(S)-1)}}{x^*(S)-1}} \enspace .
\end{equation}

\subsubsection{Rounding \texorpdfstring{$x^*$}{x*} to an Integer Point}
Consider the function $\conc \colon \R_+ \to \R_+$ defined by
\[
\conc(t) \coloneqq \frac{1-e^{-t}}{t}.
\]
and the last part of \eqref{eq:expectation1} 
\begin{equation} \label{eq:second-part-for-rounding}
  \rpiece(x) \coloneqq \sum_{S\subseteq E}(-1)^{|S|+\rgir-n}{|S|-1 \choose n-\rgir}\frac{1-e^{-(x(S)-1)}}{x(S)-1}
  = \sum_{S\subseteq E}(-1)^{|S|+\rgir-n}{|S|-1 \choose n-\rgir}\conc(x(S)-1) \enspace .
\end{equation}
as a function on $[0,1]^n$. 
The next observation is well-known, and underpins the pipage rounding technique by Ageev and Sviridenko \cite{journals/jco/AgeevS04}. 
For the sake of completeness, we include a proof in the appendix.

\begin{restatable}{observation}{rounding} \label{obs:rounding-local-concave}
  Let $f\colon [0,1]^n\to \R$ be a function such that for any $x\in [0,1]^n$ and $i,j\in [n]$,
  \[f^x_{ij}(t)\coloneqq f(x+t(e_i - e_j)) \]
  is a concave function on the domain $\{t\in [-1,1]$: $x+t(e_i- e_j)\in [0,1]^n\}$. 
  Then, for any $y\in [0,1]^n$ where $\1^{\top}y \in \Z$, there exists an integral $z\in \{0,1\}^n$ such that $f(y)\geq f(z)$ and $\1{^\top}y = \1^{\top}z$.

\end{restatable}

We would like to round $x^*$ to a binary vector using Observation~\ref{obs:rounding-local-concave}.
Hence, we need to prove concavity of $\psi$ along the directions $e_i - e_j$. 
Taking the second partial derivatives of \eqref{eq:second-part-for-rounding} yields 

\begin{equation} \label{eq:partial-derivatives-sum}
  \frac{\partial^2\rpiece}{\partial x_i\partial x_j}(x)
  = \sum_{\substack{S\subseteq E:\\i,j\in S}}(-1)^{|S|+\rgir-n} {|S|-1\choose n-\rgir} \rho''(x(S)-1) \enspace .
\end{equation}

The following claim provides a closed-form expression for all the derivatives of $\rho$, and highlights the alternating behaviour of their signs.

\begin{restatable}{claim}{derivatessingle}  \label{lem:derivatives-single-expo}
For any $k\in \Z_+$, the $k$th derivative of $\rho$ is given by 
\[
\rho^{(k)}(t) = (-1)^kk!\pr{\frac{1-e^{-t}\sum_{i=0}^k t^i/i!}{t^{k+1}} } \enspace .
\]
Consequently, if $k$ is even, then $\rho^{(k)}(t)> 0$ for all $t\geq 0$.
Otherwise, $\rho^{(k)}(t)< 0$ for all $t\geq 0$.
\end{restatable}

For the proof of concavity, we need the following notion of finite difference.

\begin{definition}
Given a function $\finitediff \colon \R\rightarrow \R$ and a scalar $x\in \R_+$, the \emph{forward difference} of $\finitediff(t)$ is
\[
\Delta_x[\finitediff](t) := \finitediff(t+x) - \finitediff(t) \enspace .
\]
More generally, for a vector $(x_1,\dots,x_n) = (x_1,\tilde{x}) \in \R^n_+$, the \emph{$n$th-order forward difference} of $\finitediff(t)$ is
\[
\Delta_x[\finitediff](t) \coloneqq \Delta_{x_1}[\Delta_{\tilde{x}}[\finitediff]](t) = \Delta_{x_1}[\Delta_{x_2}[\cdots\Delta_{x_n}[\finitediff]\cdots]](t) \enspace .
\]
\end{definition}

\begin{restatable}{claim}{highdifferences} \label{lem:higher-order-differences}
For any function $\phi\colon \R\to \R$ and vector $x\in \R^n_+$, we have
\begin{equation}\label{eq:higher-differences}
  \Delta_x[\finitediff](t) = \sum_{S\subseteq [n]} (-1)^{n-|S|}\finitediff(t+x(S)).
\end{equation}
\end{restatable}

Therefore, in the definition of $\Delta_x$, the order in which the forward difference operators $\{\Delta_{x_i}:i\in [n]\}$ are applied does not matter.
The next claim relates the signs of $\finitediff^{(n)}$ and $\Delta_x[\finitediff]$.

\begin{restatable}{claim}{differencesderivates} \label{lem:difference}
Let $\finitediff \colon \R \to \R$ be an $n$-times differentiable function.
For any $x\in \R^n_+$ and $t\in\R$, if $\finitediff^{(n)}(s)\geq 0$ for all $t\leq s\leq t+\1^{\top}x$, then $\Delta_x[\finitediff](t)\geq 0$.
\end{restatable}

We are now ready to show concavity.

\begin{lemma}\label{lem:concave}
The function $\rpiece(x)$ is concave along the direction $e_a - e_b$ for all $a,b\in E$.
\end{lemma}

\begin{proof}
  Fix $a,b\in E$ and consider the function $\phi(t) \coloneqq \rpiece(x+t(e_a - e_b))$, obtained by restricting $\rpiece$ along the direction $e_a - e_b$. %
  By substituting $y \coloneqq x+t(e_a-e_b)$ and applying the chain rule, the second derivative of $\phi(t)$ is given by 
  \[\phi''(t) = \frac{d}{dt} \pr{\sum_{i\in E} \frac{\partial \psi}{\partial y_i} \frac{dy_i}{dt}} = \frac{d}{dt} \pr{\frac{\partial \psi}{\partial y_a} - \frac{\partial \psi}{\partial y_b}} = \sum_{i\in E} \pr{\frac{\partial^2 \psi}{\partial y_a\partial y_i} - \frac{\partial^2 \psi}{\partial y_b\partial y_i}} \frac{dy_i}{dt} = \frac{\partial^2 \rpiece}{\partial y_a^2} + \frac{\partial^2 \rpiece}{\partial y_b^2}  - 2 \frac{\partial^2\rpiece}{\partial y_a\partial y_b}.\]
  By \eqref{eq:partial-derivatives-sum}, this is equal to
  \begin{align*}
    \phi''(t) &= \sum_{\substack{S\subseteq E:\\a\in S}}(-1)^{|S|+\rgir-n} {|S|-1\choose n-\rgir} \rho''(y(S)-1) + \sum_{\substack{S\subseteq E:\\b\in S}}(-1)^{|S|+\rgir-n} {|S|-1\choose n-\rgir} \rho''(y(S)-1) \\ %
    &\qquad -2\sum_{\substack{S\subseteq E:\\a,b\in S}}(-1)^{|S|+\rgir-n} {|S|-1\choose n-\rgir} \rho''(y(S)-1) \\
    &= \sum_{\substack{S\subseteq E:\\a\in S,b\notin S}}(-1)^{|S|+\rgir-n} {|S|-1\choose n-\rgir} \rho''(y(S)-1) + \sum_{\substack{S\subseteq E:\\a\notin S,b\in S}}(-1)^{|S|+\rgir-n} {|S|-1\choose n-\rgir} \rho''(y(S)-1).
  \end{align*}
  
  We show that each of the two sums above is nonpositive.
  Let us consider the first sum; the second sum follows by symmetry.
  In the first sum, every set $S\subseteq E\setminus b$ where $a\in S$ has an associated factor $\binom{|S|-1}{n-\rgir}$.
  It can be interpreted as the number of subsets in $S\setminus a$ of size $n-\rgir$.
  By charging the term associated with $S$ to these subsets, we can rewrite the first sum as 
  \begin{align*}
    \sum_{\substack{S\subseteq E:\\a\in S,b\notin S}}(-1)^{|S|+\rgir-n} {|S|-1\choose n-\rgir} \rho''(y(S)-1) 
    = \sum_{\substack{C\subseteq E\setminus\{a,b\}:\\ |C| = n-\rgir}} \sum_{D\subseteq E\setminus(C\cup\{a,b\})} (-1)^{|D|+1} \rho''(y(C\cup D\cup a)-1).
  \end{align*}

  Hence, for a fixed set $C\subseteq E\setminus\{a,b\}$ with $|C| = n-\rgir$, it suffices to show that
  \begin{equation} \label{eq:first-sum}
      \sum_{D\subseteq E\setminus (C\cup\{a,b\})}(-1)^{|D|+1}\rho''(\alpha + y(D)) %
    \end{equation}
  is nonpositive, where we denote $\alpha \coloneqq y(C\cup a) - 1$. 
  Note that $\alpha\geq 0$ because
\[y(C\cup a) = y(E) - y(E\setminus (C\cup a)) \geq \xrk - |E\setminus(C\cup a)| = \xrk - (n-(n-\rgir+1)) = \xrk - \rgir + 1 \geq 1,\]
where the first inequality is due to $y(E) = x(E) = \xrk$ and $y\leq \1$.
  Since $|E\setminus(C\cup\{a,b\})| = n-(n-\rgir+2)=\rgir-2$, we can express \eqref{eq:first-sum} as the following $(\ell-2)$th-order forward difference
  \begin{equation} \label{eq:first-sum-reformulated}
    (-1)^{1-\rgir}\sum_{D\subseteq E\setminus (C\cup\{a,b\})}(-1)^{\rgir-2-|D|}\rho''(\alpha + y(D)) \stackrel{\eqref{eq:higher-differences}}{=} (-1)^{1-\rgir}\Delta_{y_{E\setminus(C\cup\{a,b\})}}[\rho''](\alpha) \enspace .
  \end{equation}
Recall that $\rho^{(k)}(t)>0$ for all $t\geq 0$ if $k$ is even, and $\rho^{(k)}(t)<0$ for all $t\geq 0$ if $k$ is odd by Claim~\ref{lem:derivatives-single-expo}.
As $y\in \R^{\ell-2}_+$ and $\alpha\geq 0$, applying Claim~\ref{lem:difference} yields $\Delta_{y_{E\setminus(C\cup\{a,b\})}}[\rho''](\alpha) \geq 0$ if $\ell$ is even, and $\Delta_{y_{E\setminus(C\cup\{a,b\})}}[\rho''](\alpha)\leq 0$ if $\ell$ is odd. 
In both cases, \eqref{eq:first-sum-reformulated} is nonpositive.
\end{proof}

Lemma~\ref{lem:concave} allows us to round $x^*$ according to Observation~\ref{obs:rounding-local-concave}.
In particular, there exists an integral vector $x'\in \{0,1\}^n$ such that $\rpiece(x^*)\geq \rpiece(x')$ and $x^*(E) = x'(E) = \xrk$.  
Note that $x'$ has exactly $\xrk$ ones and $n-\xrk$ zeroes; recall that $\lambda\in \Z$ by Theorem \ref{thm:locating}. 
Let $T$ be the set of elements $i\in E$ where $x'_i = 1$.
Then,
\begin{align} \label{eq:sum-rounded-point}
  \rpiece(x') \stackrel{\eqref{eq:second-part-for-rounding}}{=} \sum_{S\subseteq E}(-1)^{|S|+\rgir-n}{|S|-1 \choose n-\rgir}\frac{1-e^{-(x'(S)-1)}}{x'(S)-1}
  = \sum_{S\subseteq E}(-1)^{|S|+\rgir-n}{|S|-1 \choose n-\rgir}\frac{1-e^{-(|S\cap T|-1)}}{|S \cap T|-1} \enspace .
\end{align}

\subsubsection{Simplifications for an Integer Point}

In \eqref{eq:sum-rounded-point}, every term in the sum only depends on the cardinality of $S$ and $S \cap T$, instead of the actual set~$S$. 
This allows us to rearrange the sum based on $|S\cap T|$ ranging from $0$ to $|T| = \lambda$, and $|S \setminus T|$ ranging from $0$ to $|E\setminus T| = n-\lambda$:
\begin{align*}
  \rpiece(x') &= \sum_{i=0}^\xrk \sum_{j=0}^{n-\xrk} {\xrk \choose i} {n-\xrk \choose j} (-1)^{i+j+\rgir-n} {i+j-1 \choose n-\rgir} \frac{1-e^{-(i-1)}}{i-1} \\
  &= \sum_{i=0}^\xrk {\xrk \choose i} \frac{1-e^{-(i-1)}}{i-1} \sum_{j=0}^{n-\xrk}(-1)^{i+j+\rgir-n} {n-\xrk \choose j}{i+j-1 \choose n-\rgir}  \\
  &= \sum_{i=0}^{\xrk} {\xrk \choose \xrk -i} \frac{1-e^{-(\xrk -i-1)}}{\xrk-i-1} \sum_{j=0}^{n-\xrk}(-1)^{\rgir-i-j} {n-\xrk \choose n-\xrk - j}{n-i-j-1 \choose n-\rgir} \enspace . \tag{$\substack{i \leftarrow \xrk -i\\ j \leftarrow n-\xrk -j}$}
\end{align*}
Recall that, by our convention, a binomial coefficient is zero if the upper part is smaller than the lower part.
So, by the last binomial coefficient above, we may restrict to $n-\rgir \leq n -i -j -1 \leq n-i-1$, which is equivalent to $i \leq \rgir - 1$ and $j \leq \rgir -1 -i$.
This yields
\begin{align*}
  \rpiece(x') &= \sum_{i=0}^{\rgir-1} {\xrk \choose i} \frac{1-e^{-(\xrk -i-1)}}{\xrk-i-1} \sum_{j=0}^{\rgir-1-i}(-1)^{\rgir-i-j} {n-\xrk \choose j}{n-i-j-1 \choose \rgir-i-j-1}.
\end{align*}
Note that we introduced additional terms to the inner sum if $n - \xrk < \rgir -1 -i$, but they are all $0$ due to the binomial coefficients $\binom{n-\lambda}{j}$. 
Applying Claim~\ref{clm:binomk} with $j \leftarrow \rgir -1 -i$, $k \leftarrow \xrk -1 -i$, $n \leftarrow n -1 -i$ to the inner sum gives $(-1)^{\ell-i}{\xrk-i-1 \choose \rgir-i-1 }$, leading to
\begin{align}
  \rpiece(x') =  \sum_{i=0}^{\rgir-1} (-1)^{\rgir-i} {\xrk \choose i} {\xrk-i-1 \choose \rgir-i-1 } \frac{1-e^{-(\xrk -i-1)}}{\xrk-i-1} \enspace . %
\end{align}

Observe that \eqref{eq:expectation1} evaluates to 0 if $\lambda = \ell$.
So, we may now assume that $\xrk > \rgir$. 
This allows us to apply the simple reformulation $\frac{1}{\xrk-i-1}{\xrk-i-1 \choose \rgir-i-1 } = \frac{1}{\xrk-i-1}\frac{(\xrk-i-1)!}{(\rgir-i-1)!(\xrk - \rgir)!} = \frac{(\xrk-i-2)!}{(\rgir-i-1)!(\xrk - \rgir -1)!}\frac{1}{\xrk-\rgir} = \frac{1}{\xrk - \rgir}{\xrk-i-2 \choose \rgir-i-1 }$ to obtain
\begin{align}
  \rpiece(x') = \frac{1}{\xrk - \rgir}\sum_{i=0}^{\rgir-1} (-1)^{\rgir-i} {\xrk \choose i} {\xrk-i-2 \choose \rgir-i-1 } \left(1-e^{-(\xrk -i-1)}\right) \enspace .
\end{align}
Applying Claim \ref{clm:binom2} with $j \leftarrow \rgir - 1$ and $n \leftarrow \xrk$ to $\sum_{i=0}^{\rgir-1} (-1)^{i} {\xrk \choose i} {\xrk-i-2 \choose \rgir-i-1 }$ gives $(-1)^{\rgir-1}\rgir$, resulting in
\begin{equation} \label{eq:simplified-summation}
  \begin{aligned}
  \rpiece(x') &= \frac{1}{\xrk - \rgir}\left(-\rgir - \sum_{i=0}^{\rgir-1} (-1)^{\rgir-i} {\xrk \choose i} {\xrk-i-2 \choose \rgir-i-1 } e^{-(\xrk -i-1)} \right) \\
  &= \frac{1}{\xrk - \rgir}\left(-\rgir + e^{-\xrk+1}\sum_{i=0}^{\rgir-1} (-1)^{\rgir-i-1} {\xrk \choose i} {\xrk-i-2 \choose \rgir-i-1 } e^{i} \right) \enspace .
  \end{aligned}
\end{equation}

\subsubsection{Further Simplifications}

To simplify the expression in~\eqref{eq:simplified-summation}, we consider the sum as a function of $x$ for $x = e$.
More precisely, given integral parameters $\xrk,\rgir>0$, we define the function $w_{\lambda,\ell}:\R\rightarrow \R$ as
\begin{equation*}
w_{\xrk,\rgir}(x) \coloneqq \sum_{i=0}^{\ell-1} (-1)^{\ell-1-i} {\xrk \choose i} {\xrk-2-i \choose \rgir-1-i}x^i \enspace .
\end{equation*}
Note that $w_{\lambda,\ell}$ is a polynomial on $\R$.

\begin{restatable}{claim}{derivatesbinomproductsum} \label{lem:derivates-binom-product-sum}
  For any integers $\xrk > \rgir$, we have $w_{\xrk,\rgir}(1) = \rgir$ and $w'_{\xrk,\rgir}(x) = \lambda w_{\xrk-1,\rgir-1}(x)$.
  In particular, $w^{(i)}_{\xrk,\rgir}(1) = \frac{\xrk!}{(\xrk -i)!}(\rgir -i)$. 
\end{restatable}

By Taylor's Theorem and Claim~\ref{lem:derivates-binom-product-sum}, we get
\begin{equation*}
w_{\xrk,\rgir}(x) = \sum_{i=0}^{\ell-1} \frac{w^{(i)}_{\xrk,\rgir}(1)}{i!} (x-1)^i = \sum_{i=0}^{\ell-1} \frac{\xrk!}{i!(\xrk-i)!} (\rgir-i)(x-1)^i = \sum_{i=0}^{\ell-1}{\xrk \choose i}(\rgir-i)(x-1)^i \enspace .
\end{equation*}
Plugging this back into~\eqref{eq:simplified-summation} gives us
\begin{equation} \label{eq:simple-H}
  \rpiece(x') = \frac1{\xrk - \rgir}\left(-\rgir + e^{-\xrk+1} w_{\xrk,\rgir}(e) \right) =
  \frac1{\xrk - \rgir}\left(-\rgir + e^{-\xrk+1} \sum_{i=0}^{\ell-1}{\xrk \choose i}(\rgir-i)(e-1)^i \right) \enspace .
\end{equation}

Therefore, the multilinear extension of $h$ at $x^*$ is lower bounded by
\begin{align}
  H(x^*) &\geq \E[h(Q(1))] \tag{by \eqref{eq:lower-bound-Poisson-process}}\\
  &\geq (\xrk - \rgir)\br{1-e^{-1} + e^{-1}\psi(x^*)} \tag{by \eqref{eq:expectation1}}\\ 
  &\geq (\xrk - \rgir)\br{1-e^{-1} + e^{-1}\psi(x')} \tag{rounding via Observation \ref{obs:rounding-local-concave} and Lemma \ref{lem:concave}}\\
  &= (\xrk - \rgir)\br{1-e^{-1} + \frac{e^{-1}}{\xrk - \rgir}\left(-\rgir + e^{-\xrk+1} \sum_{i=0}^{\rgir-1}{\xrk \choose i}(\rgir-i)(e-1)^i \right)} \tag{by \eqref{eq:simple-H}}\\
  &= (\xrk - \rgir)(1-e^{-1}) -\rgir e^{-1} + e^{-\xrk} \sum_{i=0}^{\rgir-1}{\xrk \choose i}(\rgir-i)(e-1)^i \notag\\
  &= \xrk - \rgir -\xrk e^{-1} + e^{-\xrk} \sum_{i=0}^{\rgir-1}{\xrk \choose i}(\rgir-i)(e-1)^i \enspace . \label{eq:simplified-lower-bound-H}
\end{align}

\subsection{Putting Everything Together}
We are finally ready to lower bound the correlation gap of the matroid rank function $\rk$.
Recall that we assumed that $\xrk > \rgir$ in the previous subsection. 
Combining the lower bounds in~\eqref{eq:Poisson-series-estimate} and~\eqref{eq:simplified-lower-bound-H} gives us
  \begin{align} \label{eq:final}
    \CG(r) = \frac{\Rk(x^*)}{\hat{\rk}(x^*)} &= \frac{G(x^*) + H(x^*)}{\1^{\top}x^*} \notag \\
    &\geq \frac{1}{\xrk}\br{\rgir - \sum_{i=0}^{\rgir-1}(\rgir-i)\frac{\xrk^ie^{-\xrk}}{i!} + 
     \xrk - \rgir -\xrk e^{-1} + e^{-\xrk} \sum_{i=0}^{\ell-1}{\xrk \choose i}(\rgir-i)(e-1)^i} \notag \\
    &= 1 - e^{-1} + \frac{e^{-\xrk}}{\xrk}\sum_{i=0}^{\rgir-1}(\rgir-i)\br{{\xrk \choose i} (e-1)^i - \frac{\xrk^i}{i!}}
  \end{align}
On the other hand, if $\xrk = \rgir$, then $h = 0$.
In this case, we obtain
\begin{equation*} \label{eq:final-uniform}
\CG(r) = \frac{\Rk(x^*)}{\hat{\rk}(x^*)} = \frac{G(x^*) + H(x^*)}{\1^{\top}x^*} = \frac{G(x^*)}{\ell} \stackrel{\eqref{eq:Poisson-series-estimate}}{\geq} 1 - \sum_{k=0}^{\rgir-1}\pr{1-\frac{k}{\ell}}\frac{\rgir^ke^{-\rgir}}{k!} = 1- \frac{\rgir^{\rgir-1} e^{-\rgir}}{(\rgir-1)!},
\end{equation*}
which also agrees with \eqref{eq:final} by Proposition \ref{prop:uniform-bound}.

To better understand the sum in \eqref{eq:final}, consider the function $\phi^\xi_\lambda:[\lambda]\rightarrow \R$ defined as
\begin{equation} \label{eq:fundamental-summand}
  \phi_{\lambda}^{\xi}(i) \coloneqq \xi\binom{\lambda}{i}(e-1)^i - \frac{\lambda^i}{i!} \enspace ,
\end{equation}
with parameters $\xi \in \R_+$ and $\lambda\in \N$.
The next claim illustrates the behaviour of $\phi^\xi_\lambda$ when $\xi > 1/(e-1)$.

\begin{restatable}{claim}{termnonnegativemonotonicity} \label{lem:term-nonnegative-monotonicity}
Given parameters $\xi > \frac{1}{e-1}$ and $\lambda\in \N$, the function $\phi_\lambda^\xi$ satisfies the following properties:
\begin{enumerate}[label=(\alph*)]
  \item If $1 \leq i\leq (\frac{e-2}{e-1})\lambda+1$, then $\phi_{\lambda}^{\xi}(i) > 0$. 
  \item\label{itm:stays-negative} If $\phi_{\lambda}^{\xi}(i) \leq 0$, then $\phi_{\lambda}^{\xi}(i+1) < 0$.
\end{enumerate}
  \end{restatable}

Applying Claim~\ref{lem:term-nonnegative-monotonicity} with $\xi = 1$ allows us to show that the bound in Theorem~\ref{thm:monster} is strictly greater than $1-1/e$. %

\begin{restatable}{lemma}{nonnegative}\label{lem:nonnegative}
For every $\lambda, \ell\in \N$ such that $\lambda \geq \ell$, we have
\[
\sum_{i=0}^{\rgir-1} (\rgir-i) \br{{\xrk \choose i}(e-1)^i - \frac{\xrk^i}{i!}} > 0 \enspace .
\]
\end{restatable}

\subsection{Monotonicity}

To complete the proof of Theorem~\ref{thm:monster}, recall that $\xrk\le\mrk$.
Hence, we need to show that the expression in \eqref{eq:final} is monotone decreasing in $\xrk$.
We derive a stronger statement, noting that the bound is $g(\xrk,\rgir)/\xrk$.

\begin{lemma}\label{lem:final-monotone}
  For any fixed $\rgir \in \N$, the expression
  \[
  g(\xrk,\ell) \coloneqq {e^{-\xrk}} \sum_{i=0}^{\rgir-1} (\rgir-i) \br{{\xrk \choose i}(e-1)^i - \frac{\xrk^i}{i!}} \,,
  \]
  is monotone decreasing in $\xrk$.
\end{lemma}
We will use the following properties of the Poisson distribution.
Their proofs are given in the Appendix.
For $k\in \Z_+$ and $x>0$, let us denote 
\begin{equation} \label{eq:Poisson-probability}
\PP_k(x) \coloneqq \Pr(\Poi(x)\le k)= e^{-x}\sum_{i=0}^k \frac{x^i}{i!} \enspace .
\end{equation}
\begin{restatable}{claim}{gammaconvex}\label{lem:Gamma-convex}
  For any fixed $k\in \Z_+$, $\PP_k(x)$ is monotone decreasing with derivative $\theta_k'(x) = -\frac{e^{-x} x^k}{k!}$.
  Furthermore, $\PP_k(x)$ is convex on the interval $(k,\infty)$. 
\end{restatable}

\begin{restatable}{claim}{poissioncdf}\label{lem:poisson-cdf}
For every $\lambda \in \N$, we have $\PP_{\lambda+1}(\lambda+1)\le \PP_{\lambda}(\lambda)$.
\end{restatable}

With these tools we are ready to prove monotonicity. 

\begin{proof}[Proof of Lemma~\ref{lem:final-monotone}]
We first prove the cases $\ell\in \{1,2,3\}$ separately: 
\[
g(\lambda,1) = 0\, , \qquad \quad g(\lambda,2) = e^{-\lambda}\lambda(e-2)\, , \qquad \quad g(\lambda,3) = \frac{e^{-\lambda}}{2}\br{e(e-2)\lambda^2 - (e-3)^2\lambda} \enspace .
\]
Their derivatives are given by
\[g'(\lambda,2) = e^{-\lambda}(-\lambda+1)(e-2)\, , \qquad \quad g'(\lambda,3) = \frac{e^{-\lambda}}{2}\br{-e(e-2)\lambda^2 + ((e-3)^2 + 2e(e-2))\lambda - (e-3)^2}\, .\]
It is easy to check that $g'(\lambda,2)<0$ for $\lambda \geq 2$, and $g'(\lambda,3)<0$ for $\lambda \geq 3$.
Henceforth, we will assume that $\ell\geq 4$.

The inequality $g(\xrk+1,\ell)< g(\xrk,\ell)$ can be reformulated as
\begin{equation} \label{eq:monotonicity-explicit}
   \begin{aligned}
   &\sum_{i=0}^{\rgir-1} (\rgir-i) \br{{\xrk+1 \choose i}(e-1)^i - \frac{(\xrk+1)^i}{i!}}< e \sum_{i=0}^{\rgir-1} (\rgir-i) \br{{\xrk \choose i}(e-1)^i - \frac{\xrk^i}{i!}} \quad \iff \\
   &\sum_{i=0}^{\rgir-1} (\rgir-i) \br{e\frac{\xrk^i}{i!} - \frac{(\xrk+1)^i}{i!}}< e\sum_{i=0}^{\rgir-1} (\rgir-i){\xrk \choose i}(e-1)^i -  \sum_{i=0}^{\rgir-1} (\rgir-i){\xrk+1 \choose i}(e-1)^i \enspace . 
   \end{aligned}
\end{equation}
For the RHS, using ${\xrk+1 \choose i}={\xrk\choose i}+{\xrk \choose i-1}$, we get
\[
\begin{aligned}
 &\sum_{i=0}^{\rgir-1} (\rgir-i) {\xrk \choose i}(e-1)^{i+1}- \sum_{i=1}^{\rgir-1} (\rgir-i) {\xrk \choose i-1}(e-1)^{i}\\
=& \sum_{i=0}^{\rgir-1} (\rgir-i) {\xrk \choose i}(e-1)^{i+1}- \sum_{i=0}^{\rgir-2} (\rgir-i-1) {\xrk \choose i}(e-1)^{i+1}\\
=&\sum_{i=0}^{\rgir-1} {\xrk \choose i}(e-1)^{i+1}
\end{aligned}
\]
Using the definition of $\PP_i(\xrk)$ from~\eqref{eq:Poisson-probability}, the LHS equals
\[
e^{\xrk+1}\sum_{i=0}^{\rgir-1} \PP_i(\xrk)-\PP_i(\xrk+1)
\]

For every $i=0,\ldots,\rgir-1$, we have $\xrk>i$.
Therefore, using the convexity and derivative of $\theta_i(x)$ from Claim~\ref{lem:Gamma-convex} leads to
\[
\PP_i(\xrk)- \PP_i(\xrk+1)\le \PP'_i(\xrk)(\xrk-(\xrk+1))=\frac{e^{-\xrk} \xrk^{i}}{i!} \enspace .
\]
Hence, \eqref{eq:monotonicity-explicit} follows by showing  
\begin{equation}\label{eq:mustbetrue}
0<\sum_{i=0}^{\rgir-1} {\xrk \choose i}(e-1)^{i+1}-\frac{e\xrk^{i}}{i!} \enspace .
\end{equation}
For the sake of brevity, we denote
\[
\varphi_\lambda(i) = {\lambda \choose i}(e-1)^{i+1} - \frac{e\lambda^i}{i!} \enspace .
\]
Then, our goal is to show that $\sum_{i=0}^{\ell-1}\varphi_\lambda(i) > 0$.

Observing that $\varphi_\lambda(i) = e\left(\frac{e-1}{e} {\xrk \choose i}(e-1)^{i}-\frac{\xrk^{i}}{i!}\right)$, we will apply Claim~\ref{lem:term-nonnegative-monotonicity} with $\xi = \frac{e-1}{e} > \frac{1}{e-1}$.
Consider the following two cases:

\medskip

\emph{Case 1:} $\varphi_\lambda(\ell-1) \geq 0$.
In this case, $\varphi_\lambda(i) \geq 0$ for all $0<i<\ell$ by Claim~\ref{lem:term-nonnegative-monotonicity}~\ref{itm:stays-negative}.
Since $\lambda \geq \ell\geq 4$,
\begin{align*}
  \sum_{i=0}^{\ell-1}\varphi_\lambda(i) &= \varphi_\lambda(0) + \varphi_\lambda(1) + \varphi_\lambda(2) + \sum_{i=3}^{\ell-1}\varphi_\lambda(i) \\
  &= -1 + \lambda((e-1)^2-e) +  \frac{\lambda}{2}\br{\lambda((e-1)^3-e) - (e-1)^3} + \sum_{i=3}^{\ell-1}\varphi_\lambda(i)> \sum_{i=3}^{\ell-1}\varphi_\lambda(i) \geq 0 \enspace .
\end{align*}

\emph{Case 2:} $\varphi_\lambda(\ell-1) < 0$. In this case, $\varphi_\lambda(i) < 0$ for all $\ell\leq i\leq \lambda$ by Claim~\ref{lem:term-nonnegative-monotonicity}~\ref{itm:stays-negative}. Thus,
\begin{align*}
  \sum_{i=0}^{\ell-1}\varphi_\lambda(i) > \sum_{i=0}^{\lambda}\varphi_\lambda(i) &= (e-1)\sum_{i=0}^{\lambda} {\xrk \choose i}(e-1)^{i}-\sum_{i=0}^{\lambda}\frac{e\xrk^{i}}{i!} = (e-1)e^\lambda - e\sum_{i=0}^\lambda \frac{\lambda^i}{i!} \\
  &= e^{\lambda+1}\pr{1-\frac1e - \PP_{\lambda}(\lambda)} \stackrel{\text{Clm.~\ref{lem:poisson-cdf}}}{\ge} e^{\lambda+1}\pr{1-\frac1e - \theta_4(4)} > 0. 
\end{align*}
The second inequality holds due to Claim~\ref{lem:poisson-cdf} together with the assumption $\lambda \geq \ell\geq 4$, whereas the last inequality follows from $1-1/e-\Pr(\Poi(4)\leq 4) > 0$.
\end{proof}

\section{Upper Bounds on the Correlation Gap}\label{sec:upper}

Let $r$ be the rank function of a matroid $\mathcal{M}$ with rank $\rho$ and girth $\gamma$.
Recall that for uniform  matroids ($\mgir=\mrk+1$), the lower bound in Theorem~\ref{thm:monster} simplifies to $1-\frac{e^{-\mrk}\mrk^{\mrk}}{\mrk!}$, and in this case it is tight.
We now give simple upper bounds on $\CG(r)$ in terms of $\rho$ and $\gamma$.
We start with the simple observation that the correlation gap of a uniform rank-$(\gamma-1)$ matroid gives such an upper bound.
\begin{observation}\label{obs:girth-uniform}
For every $\mrk,\rgir\in\N$ where $\mrk\ge \rgir$, there
 exists a matroid $\mathcal{M}=(E,\mathcal{I})$ with rank $\mrk$ and girth $\rgir+1$ whose correlation gap is arbitrarily close to
 $1-\frac{e^{-\rgir}\rgir^{\rgir}}{\rgir!}$.
\end{observation}
\begin{proof}
Let $\mathcal{M}_1$ be a rank-$\ell$ uniform matroid on $n>\ell$ elements with rank function $r_1$.
Let $\mathcal{M}_2$ be a free matroid on $\rho-\ell$ elements with rank function $r_2$.
Consider the matroid $\mathcal{M} = \mathcal{M}_1 \oplus \mathcal{M}_2$ with rank function $r$.
Note that $\mathcal{M}$ has rank $\rho$ and girth $\ell+1$.
By Proposition \ref{prop:direct-sum}, $\CG(r) = \min\{\CG(r_1),\CG(r_2)\}$.
It is known that $\CG(r_1) \to 1-\frac{e^{-\ell}\ell^\ell}{\ell!}$ as $n\to \infty$ \cite[Lemma 2.2]{journals/mp/BarmanFGG22}.
So, it is left to show that $\CG(r_2) = 1$.
For any $x\in [0,1]^{\rho-\ell}$, we have $R_2(x) = \1^{\top}x$.
Moreover, $\hat{r}_2(x) = \1^{\top}x$ by Lemma \ref{lem:concave-ext-char}.
\end{proof}

We now give an asymptotically better, albeit still non-tight upper bound. 

\begin{proposition}\label{prop:upper}
For every $\mrk,\rgir\in\N$ where $\mrk\ge \rgir$, there
 exists a matroid  $\mathcal{M}=(E,\mathcal{I})$ with rank $\mrk$ and girth $\rgir+1$ whose correlation gap is at most
 $1-\frac{1}{e}+\frac{\rgir}{e\mrk}$. 
 \end{proposition}
 \begin{proof}
Let $k:=\mrk-\rgir$.
For some $n\in\N$, $n> \rgir$, let 
the ground set be $E = E_0 \sqcup E_1 \sqcup \dots \sqcup E_{k}$, where $|E_0| = \ell n$ and $|E_i| = n$ for all $i\in [k]$.
Our matroid $\mathcal{M}$ is constructed as the union of two matroids $\mathcal{M}^u$ and $\mathcal{M}^p$. The first matroid $\mathcal{M}^u = (E,\mathcal{I}^u)$ is the uniform matroid of rank $\rgir$ on ground set $E$. The second matroid $\mathcal{M}^p = (E,\mathcal{I}^p)$ is the partition matroid on ground set $E$, where each $E_i$ is a part of rank 1 for all $i\geq 1$; every element of $E_0$ is a loop in this matroid.

Matroid union is a well known matroid operation where every independent set of the union matroid is the union of two independent sets from each of the two matroids. We can write the rank function of $\cM$ as (see e.g., \cite[Corollary 42.1a]{Schrijverbook}):
\[
r(S)=\sum_{i=1}^k \min\{1,|E_i\cap S|\}+\min\left\{\ell, |E_0\cap S|+\sum_{i=1}^k \max\{0,|E_i\cap S|-1\}\right\}\, .
\]
Note that the rank of the matroid is $\rk(E)=\rgir + k=\mrk$, and the girth is $\mgir=\rgir+1$, since every $\mgir$ element set is indepedent, but  any $\mgir+1$ element subset of $E_0$ is dependent.

Let us now fix $F\subseteq E_0$, $|F|=\rgir$, and define $x$ as $x_i=1$ if $i\in F$, $x_i=0$ if $i\in E_0\setminus F$, and $x_i=1/n$ if $i\in E\setminus E_0$. It is easy to verify that $x\in\cP(r)$ as it can be written as a convex combination of $n$ bases. Thus, $\hat r(x)=x(E)=\rgir+k=\mrk$.

Let us now compute the multilinear extension $R(x)$. Let $X\subseteq E$ be a random set sampled independently according to the probabilities $x_i$. We have $F\subseteq X$ with probability 1. 
From the above rank function expression, we get
\[
r(X)=\ell+\sum_{i=1}^k \min\{1,|E_i\cap X|\}\, .
\]
Therefore,
\[
R(x) = \E\br{r(X)}=\ell+\sum_{i=1}^k \Pr[|E_i\cap X|\ge 1]=\ell+k\left(1-\left(1-\frac{1}n\right)^n \right)\to \mrk -\frac{\mrk-\mgir}{e}\, ,
\]
as $n\to \infty$.
From here, we see that 
\[
\lim_{n\to\infty}\frac{R(x)}{\hat r(x)}=1-\frac{1}{e}+\frac{\mgir}{e\mrk}\, .\qedhere
\]
\end{proof}

\section{Conclusion}
In this paper, we derived new bounds on the correlation gap of (weighted) matroid rank functions, and we gave an overview of several of its applications. 
We first showed that for a given matroid, the correlation gap of its weighted rank function is minimized under uniform weights.
Then, we gave an improved lower bound on the correlation gap of (unweighted) matroid rank functions over $1-1/e$, parameterized by the rank and girth of the matroid. 
Our work is motivated by the important role of correlation gap in constrained monotone submodular maximization, sequential posted-price mechanism and contention resolution schemes.
We also observed that the algorithms of Barman et al.~\cite{journals/mp/BarmanFGG22,conf/stacs/BarmanFF21} for concave coverage problems fall under the pipage rounding framework of Calinescu et al.~\cite{conf/ipco/CalinescuCPV07} and Shioura \cite{journals/dmaa/Shioura09}.
In particular, their work can be interpreted as bounding the correlation gap of specific weighted matroid rank functions and $M^{\natural}$-concave functions.

By Observation \ref{obs:girth-uniform} and Proposition \ref{prop:upper}, the correlation gap of a matroid rank function is upper bounded by
\[\min\left\{1-\frac{e^{-(\mgir-1)}(\mgir-1)^{\mgir-1}}{(\mgir-1)!}, 1-\frac{1}{e}+\frac{\mgir-1}{e\mrk}\right\},\]
where $\mrk$ is the rank and $\mgir$ is the girth of the matroid.
The difference between this upper bound and the lower bound given in Theorem \ref{thm:monster} motivates a further analysis to reduce this gap. 
Another direction for future work is to explore other matroid parameters for quantifying the the correlation gap.
This could be motivated by tighter bounds on the correlation gap for special matroid classes.
Finally, it remains to open to which extent our analysis can be modified to give new insights also for other (classes of) submodular functions.

\bibliographystyle{abbrv}
\bibliography{MaxCoverage}

\appendix

\section{Omitted Proofs}

\begin{lemma}\label{lem:monotoneFixedRank}
  For any fixed $\rho \in \N$, the expression 
 \[
  \zeta(\rho,\gamma) \coloneqq \sum_{i=0}^{\gamma-2} (\gamma-1-i) 
  \br{{\rho \choose i}(e-1)^i - \frac{\rho^i}{i!}} \,,
  \]
  is monotone increasing in $\gamma$ for $\gamma \le \rho+1$.
\end{lemma}
\begin{proof}
The inequality $\zeta(\rho,\gamma) < \zeta(\rho,\gamma+1)$ can be reformulated as
\begin{equation*} 
   \begin{aligned}
   \sum_{i=0}^{\gamma - 2} (\gamma-1-i) \br{{\rho \choose i}(e-1)^i - \frac{\rho^i}{i!}}
   &< \sum_{i=0}^{\gamma-1} (\gamma-i) \br{{\rho \choose i}(e-1)^i - \frac{\rho^i}{i!}} \quad \iff \\
   \sum_{i=0}^{\gamma-1}  \br{{\rho \choose i}(e-1)^i - \frac{\rho^i}{i!}} &> 0 \enspace . 
   \end{aligned}
\end{equation*}
If every term in the sum above is positive, then we are done. 
Otherwise, we apply Claim~\ref{lem:term-nonnegative-monotonicity} for $\xi=1$.
In particular, since there exists a nonpositive term, by Claim~\ref{lem:term-nonnegative-monotonicity}~\ref{itm:stays-negative} we get
\[\sum_{i=0}^{\gamma-1}  \br{{\rho \choose i}(e-1)^i - \frac{\rho^i}{i!}}\geq\sum_{i=0}^{\rho} \br{{\rho \choose i} (e-1)^i - \frac{\rho^i}{i!}} = (e-1+1)^{\rho} - \sum_{i=0}^\rho \frac{\rho^i}{i!} = e^{\rho} - \sum_{i=0}^\rho \frac{\rho^i}{i!} > 0 \enspace . \qedhere\]
\end{proof}

\concaveextoned*

\begin{proof}
Fix an $x\in [0,1]^E$.
Let $\lambda = x(E_j)$ and $\beta = \varphi(\floor{\lambda}+1)-\varphi(\floor{\lambda})$.
Based on the LP formulation of $\hat{f}_j$, it suffices to show that $(z,\alpha) := (\beta\chi_{E_j},\varphi(\floor{\lambda}) - \beta\floor{\lambda})$ is an optimal solution to \eqref{eq:concave-ext-primal}.
Note that $z^{\top}x + \alpha = \varphi(\floor{\lambda}) + \beta(\lambda - \floor{\lambda}) = \hat{\varphi}(\lambda)$.
Feasibility is straightforward because for any $T\subseteq E$, we have
\[z(T) + \alpha = \beta\chi^{\top}_{E_j}\chi_T + \varphi(\floor{\lambda}) - \beta\floor{\lambda} = \varphi(\floor{\lambda}) + \beta(|T\cap E_j| - \floor{\lambda}) \geq \varphi(|T\cap E_j|) = f_j(T),\]
where the inequality follows from the concavity of $\hat{\varphi}$, and the fact that $\beta$ is a supergradient of $\hat{\varphi}$ at $\floor{\lambda}$.
Observe that the inequality is tight if $|T\cap E_j|\in \{\floor{\lambda},\floor{\lambda+1}\}$.
To show optimality, we consider the dual LP \eqref{eq:concave-extension}. 
By complementary slackness, it is left to prove that $x$ be can be written as a convex combination of the indicator vectors of these sets. 
Define the polytope 
\[P := \{y\in [0,1]^E: \floor{\lambda} \leq y(E_j) \leq \floor{\lambda}+1 \}.\]
It is easy to see that the vertices of $P$ are precisely the aforementioned indicator vectors.
Since $x\in P$, it lies in their convex hull.
\end{proof}

\symmetric*
\begin{proof}
Fix an $x\in [0,1]^E$, and let $\bar{x}\in [0,1]^E$ be the vector as defined above.
Let $y^* \in \argmin_{y\in [0,1]^E}\{F_j(y):y(E_j) = x(E_j)\}$ such that $\|y^*-\bar{x}\|_1$ is minimized.
It suffices to prove that $y^* = \bar{x}$.
Note that $y^*_i = \bar{x}_i$ for all $i\notin E_j$ because these coordinates do not affect the value of $F_j$.

For the purpose of contradiction, suppose that there exist $a,b\in E_j$ such that $y^*_a<\bar{x}_a$ and $y^*_b>\bar{x}_b$.
Let $\phi(t) := F_j(y^* + t(e_a - e_b))$ be the function obtained by restricting $F_j$ along the direction $e_a- e_b$ at $y^*$.
Since $f_j$ is submodular, $\phi$ is convex by Proposition \ref{prop:multilinear-convex}.
Moreover, $\phi(0) = \phi(y^*_b - y^*_a)$ because $F_j$ becomes a symmetric polynomial after fixing the coordinates in $E\setminus E_j$.
It follows that $\phi(t)\leq \phi(0) = F_j(y^*)$ for all $0\leq t\leq y^*_b- y^*_a$.
Thus, if we pick $t = \min\{\bar{x}_a - y^*_a,y^*_b - \bar{x}_b\}$, then $\|y^* + t(e_a - e_b) - \bar{x}\|_1< \|y^* - \bar{x}\|_1$, which is a contradiction.
\end{proof}

\uniformbound*
\begin{proof}
First, observe that
\begin{align*}
  \sum_{i=0}^{\ell-1} (\ell-i) {\ell \choose i}(e-1)^i &= \ell \sum_{i=0}^{\ell-1} {\ell \choose i}(e-1)^i - \sum_{i=1}^{\ell-1} \frac{\ell!}{(i-1)!(\ell-i)!}(e-1)^i \\
  &= \ell \pr{\sum_{i=0}^{\ell-1} {\ell-1 \choose i}(e-1)^i + \sum_{i=1}^{\ell-1} {\ell-1 \choose i-1}(e-1)^i - \sum_{i=1}^{\ell-1} {\ell-1 \choose i-1}(e-1)^i} = \ell e^{\ell-1}.
\end{align*}
Similarly, we have
\[\sum_{i=0}^{\ell-1} (\ell-i) \frac{\ell^i}{i!} = \ell \sum_{i=0}^{\ell-1} \frac{\ell^i}{i!} - \sum_{i=1}^{\ell-1} \frac{\ell^i}{(i-1)!} = \ell \pr{\sum_{i=0}^{\ell-1} \frac{\ell^i}{i!} - \sum_{i=0}^{\ell-2} \frac{\ell^i}{i!}} = \frac{\ell^{\ell}}{(\ell-1)!}.\]
Putting them together yields
\[1-e^{-1} + \frac{e^{-\rgir}}{\rgir} \pr{\sum_{i=0}^{\rgir-1} (\rgir-i) \br{{\rgir \choose i}(e-1)^i - \frac{\rgir^i}{i!}}} = 1-e^{-1} + \frac{e^{-\ell}}{\ell} \pr{\ell e^{\ell-1} - \frac{\ell^\ell}{(\ell-1)!}} = 1-\frac{e^{-\rgir}\rgir^{\rgir}}{\rgir!}\]
as desired.
\end{proof}

\rounding*
\begin{proof}
We proceed by strong induction on the number $k$ of non-integral coordinates in $y$.
The base case $k=0$ is trivial by picking $z=y$.
Suppose that there exists an $\ell\in \Z_+$ such that the statement is true for all $k\in \{0,2,3,\dots,\ell\}$.
Consider the case $k=\ell+1$.
Note that $k\neq 1$ because $\1^{\top}y\in \Z$.
Let $i,j\in [n]$ be distinct indices such that $y_i,y_j\in (0,1)$.
Since $f^y_{ij}$ is concave, for all $t\geq 0$ or for all $t\leq 0$, we have $f^y_{ij}(t)\leq f^y_{ij}(0)$. 
Let $\varepsilon' = \min\{1-y_i,y_j\}$ and $\varepsilon'' = \min\{y_i,1-y_j\}$, along with their corresponding points $y' = y + \varepsilon'(e_i-e_j)$ and $y'' = y - \varepsilon''(e_i - e_j)$.
Then,
\[
  \min\{f(y'),f(y'')\} = \min\{f^y_{ij}(\varepsilon'),f^y_{ij}(-\varepsilon'')\} \leq f^y_{ij}(0) = f(y) \enspace .
\]
Let $k'$ and $k''$ be the number of non-integral coordinates in $y'$ and $y''$ respectively. Note that $k',k''\neq 1$ because $\1^{\top}y' = \1^{\top}y'' = \1^{\top}y \in \Z$.
Since $k',k'' \leq \ell$, by the inductive hypothesis there exist integral $z',z''\in \{0,1\}^n$ such that $f(z') \leq f(y')$ and $f(z'') \leq f(y'')$.
Thus, our desired $z\in \{0,1\}^n$ can be chosen as
\[
z = \argmin_{x\in \{z',z''\}}f(x) \enspace . \qedhere
\] 
\end{proof}

\derivatessingle*
\begin{proof}
We prove the first part by induction on $k\geq 0$.
The base case $k=0$ is clear.
For the inductive step,
\begin{align*}
  \rho^{(k+1)}(t) &= (-1)^kk! \cdot\frac{\pr{e^{-t}\sum_{i=0}^k t^i/i! - e^{-t}\sum_{i=0}^{k-1}t^i/i!}t^{k+1} - \pr{1-e^{-t}\sum_{i=0}^k t^i/i!}(k+1)t^k}{t^{2k+2}}\\
  &= (-1)^kk! \cdot\frac{e^{-t} t^{k+1}/k! - \pr{1-e^{-t}\sum_{i=0}^k t^i/i!}(k+1)}{t^{k+2}} \\
  &= (-1)^k(k+1)! \cdot\frac{e^{-t} t^{k+1}/(k+1)! - 1 + e^{-t}\sum_{i=0}^k t^i/i!}{t^{k+2}} \\
  &= (-1)^{(k+1)}(k+1)! \cdot\frac{1-e^{-t}\sum_{i=0}^{k+1} t^i/i!}{t^{k+2}}
\end{align*}
as required.
For the second part, note that at $t=0$, applying L'H\^{o}pital's rule yields
\begin{align*}
  \rho^{(k)}(0) = \lim_{t\rightarrow 0}\, (-1)^kk!\pr{\frac{e^{-t}\sum_{i=0}^k t^i/i! - e^{-t}\sum_{i=0}^{k-1}t^i/i!}{(k+1)t^k}} = \lim_{t\rightarrow 0}\, (-1)^kk!\pr{\frac{e^{-t} t^k/k!}{(k+1)t^k}} = \frac{(-1)^k}{k+1}.
\end{align*}
Hence, $\rho^{(k)}(0)>0$ if $k$ is even, and $\rho^{(k)}(0)<0$ if $k$ is odd.
Now, let us rewrite $\rho^{(k)}(t)$ as
\[\rho^{(k)}(t) = (-1)^k \frac{k!e^{-t}}{t^{k+1}}\pr{e^t - \sum_{i=0}^k\frac{1}{i!}t^i}.\]
By the Maclaurin series of $e^t$, for all $t>0$, we have $\rho^{(k)}(t)>0$ if $k$ is even, and $\rho^{(k)}(t)<0$ if $k$ is odd. 
\end{proof}

\highdifferences*
\begin{proof}
  The claim follows by induction on $n$.
  For $n=1$, the formula simplifies to $\Delta_x[\phi](t) = (-1)^{1 - 1}\phi(t + x) - (-1)^{1-0}\finitediff(t)$ which holds by definition.

  So let $n > 1$ and $(x_1,\dots,x_n) = (\tilde{x},x_n) \in \R^n_+$.
  Using the induction hypothesis and linearity of the difference operator, we get
  \begin{align*}
    \Delta^n_x[\finitediff](t) &= \Delta_{\tilde{x}}[\Delta_{x_n}[\finitediff]](t)  =  \sum_{S\subseteq [n-1]} (-1)^{n-1-|S|}\left(\finitediff(t+x_n) - \finitediff(t)\right) \\
    &=\sum_{S\subseteq [n-1]} (-1)^{n-|S|-1} \finitediff(t+\tilde{x}(S) + x_n) + \sum_{S\subseteq [n-1]} (-1)^{n-|S|} \finitediff(t + \tilde{x}(S)) \enspace .
  \end{align*}
  This is a partition of the sum in \eqref{eq:higher-differences} into those sets containing $n$ and those not containing $n$. 
\end{proof}

\differencesderivates*
\begin{proof}
We proceed by induction on $n\geq 1$. 
The base case $n=1$ is clear as 
\[
0\leq \int_{t}^{t+x} \finitediff^{(1)}(s)ds = \finitediff(t+x) - \finitediff(t) = \Delta_x[\finitediff](t) \enspace .
\]
For the inductive step, let $x\in \R^n_+$, $y\in \R_+$, and assume that $\finitediff^{(n+1)}(s) \geq 0$ for all $t\leq s\leq t+\1^{\top}x+y$. 
Applying the inductive hypothesis to $\finitediff^{(1)}$, we have $\Delta_x[\finitediff^{(1)}](s)\geq 0$ for all $t\leq s\leq t+y$.
Using the linearity of the derivative applied to the representation \eqref{eq:higher-differences}, we obtain  
\[
0\leq \int_t^{t+y} \Delta_x[\finitediff^{(1)}](s) ds
\stackrel{\eqref{eq:higher-differences}}{=} \int_t^{t+y} \pr{\frac{d}{ds}\Delta_x[\finitediff](s)} ds
= \Delta_x[\finitediff](t+y) - \Delta_x[\finitediff](t) = \Delta_y[\Delta_x[\finitediff]](t) = \Delta_{(x,y)}[\finitediff](t)
\]
as desired.
\end{proof}

\derivatesbinomproductsum*
\begin{proof}
The first property follows from Claim~\ref{clm:binom2}.
For the second property,
\begin{align*}
  w'_{\xrk,\rgir}(x) &= \sum_{i=1}^{\rgir-1} (-1)^{\rgir-1-i} \frac{\xrk!}{(i-1)!(\xrk-i)!}{\xrk-2-i \choose \rgir-1-i}x^{i-1} \\
  &= \xrk \sum_{i=1}^{\rgir-1} (-1)^{\rgir-2-(i-1)} {\xrk-1 \choose i-1} {(\xrk-1)-2-(i-1) \choose (\rgir-1)-1-(i-1)}x^{i-1} \\
  &= \xrk \sum_{i=0}^{\rgir-2} (-1)^{\rgir-2-i} {\xrk-1 \choose i} {(\xrk-1)-2-i \choose (\rgir-1)-1-i}x^i = \xrk w_{\xrk-1,\rgir-1}(x).
\end{align*}

The formula for the derivatives follows by induction. 
\end{proof}

\termnonnegativemonotonicity*
\begin{proof}
  Fix parameters $\xi> \frac{1}{e-1}$ and $\lambda\in \N$. For $i\in [\lambda]$, we can write
  \begin{align*}
    \phi^\xi_\lambda(i) = \xi\binom{\lambda}{i}(e-1)^i - \frac{\lambda^i}{i!} &= \frac{1}{i!}\left(\xi \prod_{j=0}^{i-1} ((e-1)(\lambda - j)) - \lambda^i \right) \\
    &= \frac{\lambda}{i!}\left(\xi(e-1)\prod_{j=1}^{i-1} ((e-1)(\lambda - j)) - \lambda^{i-1} \right).
  \end{align*}
  To prove the first statement, note that 
  \[i \leq \pr{\frac{e-2}{e-1}}\lambda  + 1 \iff  \lambda \leq (e-1)(\lambda-i+1).\]
  Hence, $\lambda \leq (e-1)(\lambda - j)$ for all $j\in [i-1]$.
  As we also have $\xi(e-1)> 1$, it follows that $\phi_{\lambda}^{\xi}(i) > 0$.
  Next, we prove the second statement.
  Since $\phi_{\lambda}^{\xi}(i) \leq 0$, we obtain $\lambda > (e-1)(\lambda-i+1)$ by the first statement.
  Therefore,
  \begin{equation*}
    \begin{aligned}
      \phi_{\lambda}^{\xi}(i+1) &= \frac{\lambda}{(i+1)!}\left(\xi(e-1)\prod_{j=1}^{i} ((e-1)(\lambda - j)) - \lambda^{i} \right) \\
      &= \frac{\lambda^2}{(i+1)!}\left(\xi(e-1)\frac{(e-1)(\lambda-i)}{\lambda}\prod_{j=1}^{i-1} ((e-1)(\lambda - j)) - \lambda^{i-1} \right) <  \frac{\lambda^2}{(i+1)!} \cdot \phi_{\lambda}^{\xi}(i) \leq 0 \enspace .
    \end{aligned}
  \end{equation*}
\end{proof}

\nonnegative*
\begin{proof}
We fix $\lambda$ and apply Claim~\ref{lem:term-nonnegative-monotonicity} with $\xi = 1$. 
If $\frac{\lambda}{\lambda - \ell} < e-1$, then all summands are positive and we are done.

Otherwise, if there is a $k \leq \lambda$ such that $\phi_{\lambda}^{\xi}(k) \leq 0$, then we get for $k < j \leq \lambda$
\begin{align*}
  \sum_{i = 0}^{j}\phi_{\lambda}^{1}(i) \geq \sum_{i = 0}^{\lambda} \phi_{\lambda}^{1}(i) =
  \sum_{i=0}^\lambda \br{{\xrk \choose i} (e-1)^i - \frac{\xrk^i}{i!}} = (e-1+1)^{\xrk} - \sum_{i=0}^\lambda \frac{\xrk^i}{i!} = e^{\xrk} - \sum_{i=0}^\lambda \frac{\xrk^i}{i!} > 0 \enspace .
\end{align*}
In particular, this entails
\begin{align*}
   \sum_{i=0}^{\rgir-1} (\rgir-i) \br{{\xrk \choose i}(e-1)^i - \frac{\xrk^i}{i!}}
   = \sum_{i=0}^{\rgir-1} (\ell - i) \phi_{\lambda}^{1}(i)
   = \sum_{j=0}^{\rgir-1} \sum_{i = 0}^{j} \phi_{\lambda}^{1}(i) > 0,
\end{align*}
which concludes the proof. 
\end{proof}

\gammaconvex*
\begin{proof}
Using $\frac{d}{dx} \left(e^{-x}\frac{x^i}{i!}\right) = -e^{-x}\frac{x^i}{i!}+e^{-x}\frac{x^{i-1}}{(i-1)!}$, the derivative of $\theta_k(x)$ is 
\[
\PP'_k(x)=  -e^{-x}\sum_{i=0}^k \frac{x^i}{i!} +  e^{-x}\sum_{i=1}^k \frac{x^{i-1}}{(i-1)!} = -\frac{e^{-x} x^k}{k!}\, ,
\]
which is negative for all $x>0$. 
The second derivative of $\theta_k(x)$ is $\PP''_k(x) = e^{-x}$ if $k=0$, and
\[
\PP''_k(x)=\frac{e^{-x} x^{k-1}}{(k-1)!}\left(\frac{x}{k}- 1\right)\, 
\]
if $k\geq 1$.
In both cases, $\theta''_k(x)> 0$ when $x>k$.
\end{proof}

\poissioncdf*
\begin{proof}
Let $\Gamma \colon \R_{++}\times \R_+\rightarrow \R_+$ be the upper incomplete gamma function (see~\cite[\S 8]{OlverLozierBoisvertClark:2010}), i.e. 
\[
\Gamma(s,x) = \int_x^\infty t^{s-1}e^{-t} dt \enspace .
\]
We will use the property 
\begin{equation} \label{eq:Gamma-integral-sum}
  \frac{1}{(s-1)!}\Gamma(s,x) =  \frac{1}{(s-1)!}\int_x^\infty t^{s-1}e^{-t} dt = e^{-x}\sum_{i=0}^{s-1}\frac{x^i}{i!} \enspace ,
\end{equation}
which holds for $s\in \N$ and follows by iterated integration by parts. 

\medskip

To show the nonnegativity of $\PP_{\lambda}(\lambda) - \PP_{\lambda+1}(\lambda+1)$, recall the definition \eqref{eq:Poisson-probability} 
\begin{align*}
  \PP_{\lambda}(\lambda) - \PP_{\lambda+1}(\lambda+1) = e^{-\lambda} \sum_{i=0}^{\lambda} \frac{\lambda^i}{i!} - e^{-(\lambda+1)} \sum_{i=0}^{\lambda} \frac{(\lambda+1)^i}{i!} - \frac{e^{-(\lambda+1)}(\lambda+1)^{\lambda}}{\lambda!} \enspace .
\end{align*}
Applying \eqref{eq:Gamma-integral-sum} to the two sums yields
\begin{align*}
  \PP_{\lambda}(\lambda) - \PP_{\lambda+1}(\lambda+1)
  &= \frac{1}{\lambda!}\Gamma(\lambda+1,\lambda) - \frac{1}{\lambda!}\Gamma(\lambda+1,\lambda+1) - \frac{e^{-(\lambda+1)}(\lambda+1)^{\lambda}}{\lambda!} \\
  &= \frac{1}{\lambda!}\left(\int_{\lambda}^\infty t^{\lambda}e^{-t} dt - \int_{\lambda+1}^\infty t^{\lambda}e^{-t} dt - e^{-(\lambda+1)}(\lambda+1)^{\lambda}\right) \\
  &= \frac{1}{\lambda!}\left(\int_{\lambda}^{\lambda+1} t^{\lambda}e^{-t} dt - e^{-(\lambda+1)}(\lambda+1)^{\lambda}\right) \enspace .
\end{align*}
The integrand is monotone decreasing in the interval $(\lambda, \lambda+1]$ because $\frac{d}{dt}\left(t^{\lambda}e^{-t}\right) = (\lambda t^{\lambda-1} - t^{\lambda})e^{-t} < 0$ for all $t > \lambda$.
Hence, we can lower bound the integral by the value of the integrand at $t = \lambda+1$
\[\PP_{\lambda}(\lambda) - \PP_{\lambda+1}(\lambda+1) \geq  \frac{1}{\lambda!}\left((\lambda+1)^{\lambda}e^{-(\lambda+1)} - e^{-(\lambda+1)}(\lambda+1)^{\lambda}\right) = 0 \enspace . \qedhere\]

\end{proof}

\section{Identities for Alternating Sums of Binomial Coefficients}

\begin{claim}\label{clm:binom-single}
For any $0\leq \ell\leq n$, we have
\[\sum_{k=0}^\ell (-1)^k {n\choose k} = (-1)^\ell {n-1 \choose \ell}.\]
\end{claim}

\begin{proof}
We proceed by induction on $\ell$. 
The base case $\ell=0$ is trivial.
For the inductive step, let $\ell\geq 1$.
Then,
\begin{align*}
  \sum_{k=0}^\ell (-1)^k {n\choose k} &= (-1)^\ell {n\choose \ell} + \sum_{k=0}^{\ell-1} (-1)^k {n\choose k} \\
  &= (-1)^\ell {n\choose \ell} +  (-1)^{\ell-1} {n-1 \choose \ell-1} \\
  &= (-1)^\ell \pr{{n-1\choose \ell-1} + {n-1\choose \ell}} + (-1)^{\ell-1} {n-1 \choose \ell-1} = (-1)^\ell {n-1\choose \ell}. &&\qedhere
\end{align*}
\end{proof}

\begin{claim}\label{clm:binomial}
For any $0\leq j<n$, we have
\[
\sum_{k=0}^n (-1)^{k-1-j}{n\choose k}{k-1\choose j} = 1 \enspace .
\]
\end{claim}

\begin{proof}
We proceed by induction on $n-j\geq 1$.
For the base case $n-j = 1$, we have
\[\sum_{k=0}^n (-1)^{k-1-j}{n\choose k}{k-1\choose j} = (-1)^{n-1-(n-1)}{n\choose n}{n-1\choose n-1} = 1.\]
For the inductive step, assume that $n-j>1$. Then,
\begin{align*}
\sum_{k=0}^n (-1)^{k-1-j}{n\choose k}{k-1\choose j} &= \sum_{k=0}^n (-1)^{k-1-j}\pr{{n-1\choose k-1}+{n-1\choose k}}{k-1\choose j} \\
&= \sum_{k=0}^{n-1} (-1)^{k-j}{n-1\choose k}{k\choose j} + \sum_{k=0}^{n-1} (-1)^{k-1-j}{n-1\choose k}{k-1\choose j} \\
&= \sum_{i=0}^{n-1-j} (-1)^i{n-1\choose i+j}{i+j\choose j} + \sum_{k=0}^{n-1} (-1)^{k-1-j}{n-1\choose k}{k-1\choose j} \\
&= {n-1\choose j}\sum_{i=0}^{n-1-j} (-1)^i{n-1-j\choose i} + \sum_{k=0}^{n-1} (-1)^{k-1-j}{n-1\choose k}{k-1\choose j} \\
&= \sum_{k=0}^{n-1} (-1)^{k-1-j}{n-1\choose k}{k-1\choose j} = 1.
\end{align*}
The second last equality is due to $n-1-j>0$, while the last equality is by the inductive hypothesis.
\end{proof}

\begin{claim}\label{clm:binom0}
For any $0< j\leq n$, we have
\[\sum_{i=0}^{j} (-1)^i {n \choose i} {n-i \choose j-i} = 0.\]
\end{claim}

\begin{proof}
Using ${n \choose i} {n-i \choose j-i} = \frac{n!}{i!(n-i)!}\frac{(n-i)!}{(j-i)!(n-j)!} = \frac{n!}{i!}\frac{j!}{j!}\frac{1}{(j-i)!(n-j)!} = {n \choose j} {j \choose i}$, we get 
\[\sum_{i=0}^{j} (-1)^i {n \choose i} {n-i \choose j-i} = \sum_{i=0}^{j} (-1)^i {n \choose j} {j \choose i} = {n \choose j} (1-1)^j = 0 \enspace . \qedhere\]
\end{proof}

\begin{claim}\label{clm:binomk}
For any $0\leq j\leq k \leq n$, we have
\[\sum_{i=0}^j(-1)^i {n-k \choose i}{n-i \choose j-i}  = {k \choose j}\]
\end{claim}

\begin{proof}
Let $\{A,B\}$ be a partition of $[n]$ such that $|A| = k$ and $|B| = n-k$.
In the sum, every set $S\subseteq [n]$ of size $j$ is counted $\sum_{i=0}^{|S\cap B|}(-1)^i {|S\cap B| \choose i}$ times.
If $|S\cap B|=0$, then $S$ is counted once.
Otherwise, it is counted 0 times.
Thus, every set $S\subseteq A$ of size $j$ is counted once.
\end{proof}

\begin{claim}\label{clm:binom1}
For any $0\leq j\leq n-1$, we have
\[\sum_{i=0}^{j} (-1)^i {n \choose i} {n-1-i \choose j-i} = (-1)^j.\]
\end{claim}

\begin{proof}
We proceed by induction on $j\geq 0$. The base case $j=0$ is clear as
\[(-1)^0{n\choose 0}{n-1\choose 0} = 1.\]
For the inductive step, assume that $j>0$. Then,
\begin{align*}
  \sum_{i=0}^{j} (-1)^i {n \choose i} {n-1-i \choose j-i} &= \sum_{i=0}^{j} (-1)^i {n \choose i} \pr{{n-i \choose j-i} - {n-1-i\choose j-1-i}} \\
  &= -\sum_{i=0}^{j} (-1)^i {n \choose i} {n-1-i \choose j-1-i} \tag{Claim \ref{clm:binom0}} \\
  &= -\sum_{i=0}^{j-1} (-1)^i {n \choose i} {n-1-i \choose j-1-i} = (-1)^j. \tag{Inductive hypothesis}
\end{align*}
\end{proof}

\begin{claim}\label{clm:binom2}
For any $0\leq j\leq n-2$, we have
\[\sum_{i=0}^{j} (-1)^i {n \choose i} {n-2-i \choose j-i} = (-1)^j(j+1).\]
\end{claim}

\begin{proof}
We proceed by induction on $j\geq 0$. The base case $j=0$ is clear as
\[(-1)^0{n\choose 0}{n-2\choose 0}  = 1.\]
For the inductive step, assume that $j>0$. Then,
\begin{align*}
  \sum_{i=0}^{j} (-1)^i {n \choose i} {n-2-i \choose j-i} &= \sum_{i=0}^{j} (-1)^i {n \choose i} \pr{{n-1-i \choose j-i} - {n-2-i\choose j-1-i}} \\
  &= (-1)^j-\sum_{i=0}^{j} (-1)^i {n \choose i} {n-2-i \choose j-1-i} \tag{Claim \ref{clm:binom1}} \\
  &= (-1)^j-\sum_{i=0}^{j-1} (-1)^i {n \choose i} {n-2-i \choose j-1-i} \\
  &= (-1)^j - (-1)^{j-1}j = (-1)^j(1+j). \tag{Inductive hypothesis}
\end{align*}
\end{proof}

\section{Attainment of the Correlation Gap}
\label{sec:attainment}

In this section, we address the issue of attainment of the correlation gap.
The following theorem shows that the correlation gap is always attained for a nonnegative monotone submodular function.

\begin{theorem}\label{thm:attain}
For any monotone submodular function $f:2^E\to \R_+$, there exists a point $x^*\in [0,1]^E$ such that $\CG(f) = F(x^*)/\hat{f}(x^*)$.
\end{theorem}

\begin{proof}
Since $[0,1]^E$ is compact, by Weierstrass Theorem it suffices to prove that the function $\phi:[0,1]^E\to \R_+$ defined by
\[\phi(x)\coloneqq\begin{cases}
  F(x)/\hat{f}(x) &\text{ if }\hat{f}(x)\neq 0,\\
  1, &\text{ otherwise}.
\end{cases}\]
is continuous.
As $F$ and $\hat{f}$ are both continuous, we only need to check the zeroes of $\hat{f}$.
Let $x\in [0,1]^E$ such that $\hat{f}(x) = 0$.
Then, $F(x) = 0$ because $0\leq F\leq \hat{f}$. 
Note that $f(\supp(x)) = 0$ because $f$ is nonnegative.
By the monotonicity of $f$, we also get $f(S)=0$ for all $S\subseteq \supp(x)$.

Recall the dual form~\eqref{eq:concave-ext-primal} of $\hat{f}(x)$.
We first show that the solution $(z^*,\alpha^*)$ given by $z^*_i \coloneqq f(\{i\})$ for all $i\in E$ and $\alpha^* \coloneqq 0$ is optimal.
By submodularity, we have
\[z^*(S) + \alpha^* = \sum_{i\in S}f(\{i\}) \geq f(S)\]
for all $S\subseteq E$, which proves feasibility.
Its objective value is 
\[(z^*)^\top x + \alpha^* = \sum_{i\in \supp(x)} f(\{i\})x_i =0 = \hat{f}(x),\]
so it is optimal.
Observe that the linear function $(z^*)^\top y = (z^*)^{\top}(y-x)$ is the first-order Taylor approximation of $F$ at $x$.
Indeed, letting $X$ denote the random set obtained by picking each element $j\in E\setminus \{i\}$ independently with probability $x_j$, by Proposition~\ref{prop:multilinear-gradient} the first-order Taylor approximation of $F$ at $x$ is
\begin{align*}
  F(x) + \nabla F(x)^\top (y-x) &= \sum_{i\in E} \E[f(X+i) - f(X)](y_i-x_i) \\
  &= \sum_{i\in E}\E[f(X+i)](y_i-x_i) = \sum_{i\in E} f(\{i\})(y_i-x_i) = (z^*)^{\top} (y-x).
\end{align*}
The penultimate equality holds because $f(\{i\})\leq f(S+i)$ for all $S\subseteq \supp(x)$ by monotonicity, while the reverse inequality is given by submodularity and $f(S) = 0$ for all $S\subseteq \supp(x)$.

\begin{claim}
For every optimal solution $(z,\alpha)$ to the dual form \eqref{eq:concave-ext-primal} of $\hat{f}(x)$ and $y\in [0,1]^E$, we have 
\[(z^*)^\top y \leq z^{\top}y + \alpha.\]
\end{claim}

\begin{proof}
Let us partition $E$ into the following 3 sets
\[I_0 \coloneqq \{i\in E:x_i = 0\} \qquad \qquad I_f \coloneqq \{i\in E:0<x_i<1\} \qquad \qquad  I_1 \coloneqq \{i\in E:x_i = 1\}.\]
Fix an optimal solution $(z,\alpha)$ to \eqref{eq:concave-ext-primal}.
Since $\hat{f}(x) = z^{\top}x + \alpha = 0$ and $\hat{f}\geq 0$, we deduce that $z_i\geq 0$ for all $i\in I_0$, $z_i = 0$ for all $i\in I_f$, and $z_i\leq 0$ for all $i\in I_1$.
It suffices to show that $z_i\geq z^*_i$ for all $i\in I_0$.
This would imply that for any $y\in [0,1]^E$,
\[z^{\top}y+\alpha = z^{\top}(y-x) = \sum_{i\in I_0}z_iy_i + \sum_{i\in I_1}z_i(y_i - 1)\geq \sum_{i\in I_0}z_iy_i \geq \sum_{i\in I_0}z^*_iy_i = (z^*)^\top y .\]
For the purpose of contradiction, let $j\in I_0$ with $z_j<z^*_j$.
Then, 
\[\hat{f}(\chi_{\supp(x)+j}) \leq z^\top(\chi_{\supp(x)+j}) + \alpha = z^\top(\chi_{\supp(x)+j}-x) = z_j < z^*_j = f(\{j\}) \leq f(\supp(x)+j).\]
The first inequality is due to the feasibility of $(z,\alpha)$, whereas the last inequality uses the monotonicity of $f$. 
However, this contradicts the fact that $\hat{f}$ is an extension of $f$.
\end{proof}

Since $\hat f$ is piecewise affine, the claim above shows that $\hat{f}(y) = (z^*)^\top y$ when $y$ is sufficiently close to $x$.
As $F$ is 2-times continuously differentiable, Taylor's Theorem tells us that for every $y$, we have 
\[F(y) = (z^*)^\top y + \frac12 (y-x)^\top H(c_y) (y-x),\]
where $H(c_y)$ is the Hessian of $F$ evaluated at some $c_y\in [x,y]$.
Combining these two facts yields
\begin{equation}\label{eq:ratio-converge}
  \lim_{\substack{y\to x:\\\hat f(y)\neq 0}} \frac{F(y)}{\hat{f}(y)} = \lim_{\substack{y\to x:\\(z^*)^\top y \neq  0}} \frac{(z^*)^\top y+\frac12 (y-x)^{\top}H(c_y)(y-x)}{(z^*)^\top y} = 1+ \frac12 \lim_{\substack{y\to x:\\(z^*)^\top y\neq 0}} \frac{(y-x)^{\top}H(c_y)(y-x)}{(z^*)^\top y}.
\end{equation}
It is left to prove that the second term on the RHS of \eqref{eq:ratio-converge} converges to 0.
By Proposition~\ref{prop:multilinear-hessian},
\[H(c_y)_{ij} = \frac{\partial^2F}{\partial x_i\partial x_j}(c_y) = \E[f(C_y+i+j) - f(C_y+i) - f(C_y+j) + f(C_y)]\leq 0,\]
where $C_y$ is the random set obtained by picking each element $k\in E\setminus\{i,j\}$ independently with probability $(c_y)_k$.
Moreover, if $f(\{i\}) = 0$ for some $i\in E$, then $f(S+i) - f(S) = 0$ for all $S\subseteq E$ by monotonicity and submodularity. 
It follows that $H(c_y)_{ij} = H(c_y)_{ji} = 0$ for all $j\in E$ if $z^*_i = 0$.
Since $z^*_i = 0$ for all $i\in \supp(x)$, letting $N \coloneqq E\setminus \supp(x)$, we have
\begin{align}\label{eq:error-converge}
0\geq \frac{(y-x)^{\top}H(c_y)(y-x)}{(z^*)^\top y} &= \frac{y_N^\top H(c_y)_{N,N}y_N}{(z^*_N)^\top y_N} \notag \\
&\geq \sum_{i\in N:z^*_iy_i>0} \frac{y_i\sum_{j\in N}H(c_y)_{ij}y_j}{z^*_iy_i} \geq \sum_{i\in N:z^*_iy_i>0} \frac{\sum_{j\in N}H(c_y)_{ij}y_j}{z^*_i}.
\end{align}
The first inequality is due to $\hat{f}(y)\leq (z^*)^{\top}y$ for all $y\in [0,1]^E$, whereas the second inequality is by $H\leq \0$ and $y,z^*\geq \0$. 
As $y\to x$, the RHS of \eqref{eq:error-converge} converges to 0 because $x_N = \0$ and the Hessian is bounded, i.e., $H(c_y)\leq 2\max f$ for all $c_y\in [0,1]^E$. 
Thus, by the squeeze theorem, the RHS of \eqref{eq:ratio-converge} converges to 0 as desired.
\end{proof}

The next two propositions show that neither the monotonicity nor submodularity assumption can be dropped without affecting the attainment of the correlation gap.

\begin{proposition}\label{prop:unattained}
There exists a monotone function whose correlation gap is not attained.
\end{proposition}

\begin{proof}
Fix $\varepsilon\in (0,1/2)$.
Consider the function $f:2^{[2]}\to \R_+$ defined by $f(\emptyset) = 0$, $f(\{1\}) = f(\{2\}) = \varepsilon$, and $f(\{1,2\}) = 1$.
Clearly, $f$ is monotone but not submodular.
The multilinear extension of $f$ is
\[F(x) = \varepsilon(x_1(1-x_2) + x_2(1-x_1)) + x_1x_2 = \varepsilon x_1 + \varepsilon x_2 + (1-2\varepsilon)x_1x_2.\]
By \eqref{eq:concave-ext-primal}, the concave extension of $f$ is
\[\hat{f}(x) = \begin{cases}
    \varepsilon x_1 + (1-\varepsilon)x_2 &\text{ if }x_1\geq x_2,\\
    (1-\varepsilon) x_1 + \varepsilon x_2 &\text{ if }x_1\leq x_2.
  \end{cases}\]
On the line $L\coloneqq \{\alpha\1:0\leq \alpha\leq 1\}\subseteq [0,1]^2$, the ratio
\[\frac{F(\alpha\1)}{\hat f(\alpha\1)} = \frac{2\varepsilon\alpha + (1-2\varepsilon)\alpha^2}{\alpha} = 2\varepsilon + (1-2\varepsilon)\alpha\]
converges to $2\varepsilon<1$ as $\alpha\to 0$.
However, $F(\0)/\hat{f}(\0) = 0/0 = 1$.
To show that $\CG(f)$ is not attained, it suffices to prove that $F(x)/\hat{f}(x)>2\varepsilon$ for all $x\in [0,1]^2\setminus L$.
In fact, it is enough to prove that for any $\alpha\1\in L$ and $\beta\in \R$,
\begin{equation}\label{eq:unattained}
  \frac{F(\alpha\1)}{\hat{f}(\alpha\1)}\leq \frac{F(\alpha\1+\beta(\chi_1-\chi_2))}{\hat{f}(\alpha\1+\beta(\chi_1-\chi_2))}.
\end{equation}
Fix a point $\alpha\1\in L$, and define the function $\phi(\beta)$ as the RHS of \eqref{eq:unattained}.
Note that $\phi(\beta) = \phi(-\beta)$ because $F$ and $\hat{f}$ are symmetric.
So, we may assume that $\beta\geq 0$ without loss of generality.
We also have $\beta\leq \min(\alpha,1-\alpha)\leq 1/2$.
Then,
\[\phi(\beta) = \frac{2\varepsilon\alpha + (1-2\varepsilon)(\alpha^2 -\beta^2)}{\alpha - (1-2\varepsilon)\beta }.\]
It is left to show that $\phi'(\beta)\geq 0$. Differentiating yields
\begin{align*}
  (\alpha-(1-2\varepsilon)\beta)^2\phi'(\beta) &= -2(1-2\varepsilon)\beta(\alpha-(1-2\varepsilon)\beta) + (1-2\varepsilon)(2\varepsilon\alpha + (1-2\varepsilon)(\alpha^2 - \beta^2))  \\
  &= -2\alpha\beta(1-2\varepsilon) + (1-2\varepsilon)^2(\alpha^2+\beta^2) + 2\varepsilon\alpha(1-2\varepsilon) \\
  &\geq -2\alpha\beta(1-2\varepsilon) + (1-2\varepsilon)^2(\alpha^2+\beta^2) + 4\varepsilon\alpha\beta(1-2\varepsilon) \\
  &= (1-2\varepsilon)^2(\alpha-\beta)^2 \geq 0.
\end{align*}
The first inequality is due to $\beta\leq 1/2$, while the last inequality follows from $\beta\leq \alpha$.
\end{proof}

\begin{proposition}
There exists a submodular function whose correlation gap is not attained.
\end{proposition}

\begin{proof}
Fix $\varepsilon\in (0,1/2)$.
Consider the function $g:2^{[2]}\to \R_+$ defined by $g(\emptyset) = \varepsilon$, $g(\{1\}) = 0$, $g(\{2\}) = 1$, and $g(\{1,2\}) = \varepsilon$.
Clearly, $g$ is submodular but not monotone.
Let $\psi:[0,1]^E\to [0,1]^E$ be the $90^\circ$-clockwise rotation map about the point $(1/2,1/2)$, i.e.,  $\psi(x) \coloneqq (x_2,1-x_1)$.
We claim that $G(x) = F(\psi(x))$ and $\hat{g}(x) = \hat{f}(\psi(x))$ for all $x\in [0,1]^E$, where $f$ is the function defined in the proof of Proposition \ref{prop:unattained}.
Indeed,
\[F(\psi(x)) = \varepsilon(x_2x_1+(1-x_1)(1-x_2)) + x_2(1-x_1) = G(x).\]
By \eqref{eq:concave-ext-primal}, the concave extension of $g$ is
\[\hat{g}(x) = \begin{cases}
  (\varepsilon-1)x_1 + \varepsilon x_2 + 1-\varepsilon &\text{ if }x_1 + x_2\geq 1,\\
  -\varepsilon x_1 + (1-\varepsilon)x_2 + \varepsilon &\text{ if }x_1 + x_2\leq 1,
\end{cases}\]
which also agrees with
\[\hat{f}(\psi(x)) = \begin{cases}
  \varepsilon x_2 + (1-\varepsilon)(1-x_1) &\text{ if }x_1 + x_2\geq 1,\\
  (1-\varepsilon)x_2 + \varepsilon(1-x_1) &\text{ if }x_1 + x_2\leq 1.
\end{cases}\]
Since $\phi$ is onto, it follows that $\CG(f) = \CG(g)$.
Thus, $\CG(g)$ is not attained by Proposition~\ref{prop:unattained}.
\end{proof}

\end{document}